\documentclass[11pt]{article}
\usepackage{fullpage,amsthm,enumitem}
\usepackage{amsmath,amsfonts,amssymb,graphicx} 

\pdfoutput=1

\usepackage{mathpazo} 
\usepackage[dvipsnames]{xcolor} 

\usepackage[all,knot]{xy}
\xyoption{arc}

\newtheorem{thm}{Theorem}
\newtheorem{cor}[thm]{Corollary}
\newtheorem{lem}[thm]{Lemma}
\newtheorem{prop}[thm]{Proposition}
\newtheorem{defn}[thm]{Definition}
\newtheorem{ex}[thm]{Example}
\newtheorem{rem}[thm]{Remark}
\def\A{\mathcal{A}}

\def\ev{\textup{ev}}
\def\odd{\textup{odd}}
\def\W{\mathcal{W}}

\def\Wsn{\W^{s,n}}
\def\C{\mathbb{C}}
\def\cs{\mathrm{cs}}

\def\E{\mathcal{E}}

\def\H{\mathcal{H}}

\def\R{\mathbb{R}}
\DeclareMathOperator{\tr}{Tr}
\DeclareMathOperator{\Tr}{Tr}

\def\S{\mathcal{S}}
\def\O{\mathcal{O}}

\def\mB{\mathcal{B}}
\renewcommand{\L}{\mathcal{L}}

\newcommand{\br}[1]{\langle#1\rangle}

\newcommand{\braket}[2]{\left\langle #1 \middle| #2 \right\rangle}
\newcommand{\vect}[2]{\begin{pmatrix}#1\\#2\end{pmatrix}}

\newcommand{\Esg}{\mathcal{E}^{s,\gamma}}

\newcommand{\N}{\mathbb N}
\newcommand{\nrm}[2]{\left\|#1\right\|_{#2}}
\newcommand{\norm}[1]{\left\|#1\right\|}
\newcommand{\supnorm}[1]{\norm{ #1 }_\infty}
\newcommand{\absnorm}[1]{\norm{ #1 }_1}

\newcommand{\sa}{\textnormal{sa}}

\newcommand{\Tfn}{T_{f^{[n]}}}

\makeatletter
\DeclareMathOperator{\exterior}{\@ifnextchar^\@exterior{\@exterior^{}}}
\def\@exterior^#1{\mathop{\bigwedge\nolimits^{\!#1}}}
\makeatother

\title{Cyclic cocycles in the spectral action}
\author{Teun D.H. van Nuland\footnote{t.vannuland@math.ru.nl},\qquad Walter D. van Suijlekom\footnote{waltervs@math.ru.nl}\\
	Radboud University, Heyendaalseweg 135,\\
	6525 AJ Nijmegen, The Netherlands}

\begin{document}

\maketitle

\begin{abstract}
	We show that the spectral action, when perturbed by a gauge potential, can be written as a series of Chern--Simons actions and Yang--Mills actions of all orders. In the odd orders, generalized Chern--Simons forms are integrated against an odd $(b,B)$-cocycle, whereas, in the even orders, powers of the curvature are integrated against $(b,B)$-cocycles that are Hochschild cocycles as well. In both cases, the Hochschild cochains are derived from the Taylor series expansion of the spectral action $\tr(f(D+V))$ in powers of $V=\pi_D(A)$, but unlike the Taylor expansion we expand in increasing order of the forms in $A$. This extends \cite{CC06}, which computes only the scale-invariant part of the spectral action, works in dimension at most 4, and assumes the vanishing tadpole hypothesis. In our situation, we obtain a truly infinite odd $(b,B)$-cocycle. The analysis involved draws from recent results in multiple operator integration, which also allows us to give conditions under which this cocycle is entire, and under which our expansion is absolutely convergent. As a consequence of our expansion and of the gauge invariance of the spectral action, we show that the odd $(b,B)$-cocycle pairs trivially with $K_1$.
\end{abstract}

\tableofcontents 
\section{Introduction}
The spectral action \cite{CC96,CC97} is one of the key instruments in the applications of noncommutative geometry to particle physics. With inner fluctuations \cite{C96} of a noncommutative manifold playing the role of gauge potentials, the spectral action principle yields Lagrangians for them. In fact, the asymptotic behavior of the spectral action for small momenta leads in this way to local field theories. For example, this provides the simplest way known to geometrically explain the dynamics and interactions of the gauge bosons and the Higgs boson in the Standard Model Lagrangian as an effective field theory \cite{CCM07} (see also the textbooks \cite{CM07,Sui14}). These techniques extend to more general noncommutative manifolds, and, indeed, if one restrict to the scale-invariant part, one may naturally identify a Yang--Mills term and a Chern--Simons term to elegantly appear in the spectral action \cite{CC06}. However, it remains an open problem to determine the general form of the spectral action without resorting to the scale-invariant part, even though from the physical perspective there is a strong desire to go beyond this effective field theory approach.

Motivated by this, we study the spectral action when it is expanded in terms of inner fluctuations associated to an arbitrary noncommutative manifold, without resorting to heat-kernel techniques. Indeed, the latter are not always available and an understanding of the full spectral action could provide deeper insight in the origin of these gauge theories from noncommutative geometry. Let us now give a more precise description of our setup.

We let $(\A,\H,D)$ be an $s$-summable spectral triple ({\em cf.} Definition \ref{defn:st} below). If $f : \R \to \C$ is a suitably nice function we may define the spectral action \cite{CC97}:
$$
\tr (f(D)).
$$
An inner fluctuation, as explained in \cite{C96}, is given by a Hermitian universal one-form
\begin{align}\label{eq:A uitgeschreven}
	A=\sum_{j=1}^n a_jdb_j\in\Omega^1(\A),
\end{align}
for elements $a_j,b_j\in\A$. The terminology `fluctuation' comes from representing $A$ on $\H$ as
\begin{align}\label{eq:V uitgeschreven}
	V:=\pi_D(A):=\sum_{j=1}^n a_j[D,b_j]\in\mB(\H)_\sa,
\end{align}
and fluctuating $D$ to $D+V$ in the spectral action.
 %
The variation of the spectral action under the inner fluctuation is then given by
\begin{align}\label{variation of SA}
	\tr(f(D+V))-\tr(f(D)).
\end{align}
As spectral triples can be understood as noncommutative spin manifolds encoding the gauge fields as an inner structure, one could hope that perturbations of the spectral action could be understood in terms of noncommutative versions of geometrical, gauge theoretical concepts. Hence we would like to express \eqref{variation of SA} in terms of universal forms constructed from $A$. To express an action functional in terms of universal forms, one is naturally led to cyclic cohomology. As it turns out, hidden inside the spectral action we will identify an odd $(b,B)$-cocycle $(\tilde\psi_1,\tilde\psi_3,\ldots)$ and an even $(b,B)$-cocycle $(\phi_2,\phi_4,\ldots)$ for which $b\phi_{2k}=B\phi_{2k}=0$, i.e., each Hochschild cochain $\phi_{2k}$ forms its own $(b,B)$-cocycle $(0,\ldots,0,\phi_{2k},0,\ldots)$. On the other hand, the odd $(b,B)$-cocycle $(\tilde\psi_{2k+1})$ is truly infinite. 

 The main result of this paper is that for suitable $f:\R\to\C$ we may expand
\begin{align}\label{eq:expansion intro}
	\tr(f(D+V)-f(D))=\sum_{k=1}^\infty\left(\int_{\psi_{2k-1}}\cs_{2k-1}(A)+\frac{1}{2k}\int_{\phi_{2k}}F^{k}\right),
\end{align}
	in which the series converges absolutely. Here $\psi_{2k-1}$ is a scalar multiple of $\tilde\psi_{2k-1}$, $F_t=tdA+t^2A^2$, so that $F=F_1$ is the curvature of $A$, and $\cs_{2k-1}(A)=\int_0^1 AF_t^{k-1}dt$ is a generalized noncommutative Chern--Simons form. We also give a bound on the remainder of this expansion.

As already mentioned, a similar result was shown earlier to hold for the scale-invariant part of the spectral action. Indeed,  Connes and Chamseddine \cite{CC06} expressed the scale-invariant part in dimension $\leq 4 $ as
\begin{equation*}
\zeta_{D+V}(0) - \zeta_D(0) = - \frac 1 4\int_{\tau_0} (dA+A^2) + \frac 12 \int_\psi \left(A d A + \frac 2 3 A^3\right),
\end{equation*}
for a certain Hochschild 4-cocycle $\tau_0$ and cyclic 3-cocycle $\psi$.

Interestingly, a key role in our extension of this result to the full spectral action will be played by multiple operator integrals. It is the natural replacement of residues in this context, and also allows to go beyond dimension $4$.
For our analysis of the cocycle structure that appears in the full spectral action we take the Taylor series expansion as a starting point. 
This explains the appearance of multiple operator integrals, as traces thereof are multilinear extensions of the derivatives of the spectral action. This viewpoint is also studied in \cite{Skr13,Sui11}, where multiple operator integrals are used to investigate the Taylor expansion of the spectral action. As we will show, multiple operator integrals can also be used to define cyclic cocycles, because of some known properties of the multiple operator integral that have been proved in increasing generality in the last decades (e.g., in \cite{ACDS09,CS18,vNS21,Skr13,Sui11}). We push these results a bit further, in Section \ref{sct:Function spaces}, by proving estimates and continuity properties for the multiple operator integral when the self-adjoint operator has an $s$-summable resolvent. Extending these results to the case of non-unital spectral triples is left open for future research.


These techniques also allow us to show that the found $(b,B)$-cocycles are {\em entire} in the sense of \cite{C88a}. This makes it meaningful to analyze their pairing with K-theory, which we find to be trivial in Section \ref{sct:vanishing pairing}.

\subsection*{Conventions}
Throughout, we fix a separable Hilbert space $\H$. If we say $D$ is self-adjoint in $\H$, it is possibly unbounded and self-adjoint with a domain dense in $\H$. We let $\mB(\H)$ be the C*-algebra of bounded operators on $\H$. For $M\subseteq\mB(\H)$, $M_\sa$ denotes the self-adjoint elements in $M$. We denote by $\mathcal{S}^s\subseteq\mB(\H)$ the Schatten class of $s$-summable operators, for $s\in\N=\{1,2,\ldots\}$.
We write $\N_0=\{0,1,2,\ldots\}$.

\paragraph{Universal forms}
When considering a spectral triple $(\A,\H,D)$, we write $\Omega^\bullet(\A)=\oplus_{n\in\N_0}\Omega^n(\A)$ for the universal differential graded algebra over $\A=:\Omega^0(\A)$, endowed with grading $d$. We write $\Omega^1_D(\A):=\pi_D(\Omega^1(\A))$ where $\pi_D:\Omega^1(\A)\to\mB(\H)$ is defined by \eqref{eq:A uitgeschreven} and \eqref{eq:V uitgeschreven}. Whenever $A\in\Omega^1(\A)$, we write $F:=dA+A^2\in\Omega^2(\A)$ for the curvature of $A$.

\paragraph{Functions} We write $C^n$ for the space of $n$ times continuously differentiable functions, $C:=C^0$, $C_0$ for the space of continuous functions vanishing at infinity, $C^n_c$ for the space of compactly supported functions in $C^n$, $\mathcal{D}:=C_c^\infty$ for the space of smooth compactly supported functions, $\S$ for the Schwartz space, and $L^p$ for the Lebesgue spaces ($p\in[1,\infty]$). All these functions are defined on $\R$ and have values in $\C$. We define $u(x):=x-i$, and write $u^{-1}$ for the multiplicative inverse of $u\in C$. 

\paragraph{Fourier transforms}We define the Fourier transform of an integrable function $f$ on $\R$ by $\hat{f}(x):=\int_\R \frac{dy}{2\pi}f(y)e^{-iyx}.$ For a general (not necessarily tempered) distribution $f\in\mathcal{D}'$ we can still define the Fourier transform as a distribution $\hat{f}:\hat{\mathcal D}\to\C$ by $\langle \hat{f}|\varphi\rangle:=\braket{f}{\hat{\varphi}}$ for all Schwartz functions $\varphi$ with $\hat{\varphi}\in\mathcal D$. The restriction $~\hat{~}:\S'\to\S'$ is bijective with inverse denoted by $\check{~}:\S'\to\S'$. 

This more general definition of the Fourier transform will be applied in the following way. For an arbitrary continuous function $f$ (which is in $\mathcal D'$ but not a priori in $\S'$) we will often assume that $\hat{f}\in L^1$. Because then $\hat{f}\in\S'$ and $\check{\hat{f}}\in C_0$ by the Riemann-Lebesgue lemma, we find that $\hat{\check{\hat{f}}}=\hat{f}$, which implies $\braket{f}{\varphi}=\langle\check{\hat{f}}|\varphi\rangle$ for all $\varphi\in\mathcal D$. Therefore $f=\check{\hat{f}}\in C_0$ and $f(x)=\int \hat{f}(y)e^{ixy}\,dy$.

\section{Preliminaries on multiple operator integrals}
In this section, we introduce the subject of multiple operator integration in a sufficiently light way to fit our purposes. For a more extensive discussion, we refer to \cite{ACDS09,PSS13,ST19}. 

For every $n\in\N_0$ and $f\in C^n$, define the divided difference $f^{[n]}$ recursively as follows:
\begin{align*}
	f^{[0]}(x_0)&:=f(x_0)\\
	f^{[n+1]}(x_0,\ldots,x_{n+1})&:=\begin{cases}\frac{f^{[n]}(x_0,\ldots,x_{n})-f^{[n]}(x_0,\ldots,x_{n-1},x_{n+1})}{x_n-x_{n+1}}&\text{if }x_n\neq x_{n+1}\\
	\frac{\partial}{\partial x_n} f^{[n]}(x_0,\ldots,x_n)&\text{if }x_n=x_{n+1}.
	\end{cases}
\end{align*}
 Let $\sigma$ denote the standard measure on the $n$-simplex, $$\Delta_n:=\left\{(s_0,\ldots,s_n)\in\R^{n+1}_{\geq0}:~ \sum_{j=0}^n s_j=1\right\}.$$  In order to define a multilinear operator integral, we need the following lemma.
\begin{lem}\label{lem:divided diff}
Whenever $f\in C^n$ is such that $\widehat{f^{(n)}}\in L^1$, we can write $f^{[n]}$ as
	$$f^{[n]}(x_0,\ldots,x_n)=\int_{\Delta_n}\int_\R e^{its_0x_0}\cdots e^{its_nx_n}\widehat{f^{(n)}}(t)\,dt\,d\sigma(s_0,\ldots,s_n).$$
%
\end{lem}
\begin{proof}
One simply combines the proofs of \cite[Lemma 5.1]{PSS13} and \cite[Lemma 5.2]{PSS13}.
\end{proof}
It is easily seen that $(\Sigma,\sigma_f):=(\Delta_n\times\R,\sigma\times \widehat{f^{(n)}})$ is a finite measure space with total variation equal to $\tfrac{1}{n!}\|\widehat{f^{(n)}}\|_1$.


\begin{defn}
	Let $D_0\ldots,D_n$ be self-adjoint in $\H$ and let  $f \in C^{n}$ such that $\widehat{f^{(n)}} \in L^1$. The {\em multiple operator integral} $\Tfn^{D_0,\ldots,D_n}:\mB(\H)^{\times n}\to\mB(\H)$ is defined on $V_1,\ldots,V_n\in\mB(\H)$ by
\begin{align}\label{def:Tfn}
	\Tfn^{D_0,\ldots,D_n}(V_1,\ldots,V_n)y:=\int_{\Delta_n}\int_\R   e^{its_0D_0}V_1e^{its_1D_1}\cdots V_n e^{its_nD_n}y\:\widehat{f^{(n)}}(t)\:dt\:d\sigma(s_0,\ldots,s_n),
\end{align}
for $y\in\H$. In the case that $D_j=D$ for all $j$, we also write $\Tfn^D:=\Tfn^{D,\ldots,D}$.
\end{defn}
This is a special case of the operator integral $T_\phi$ in \cite[Definition 4.1]{ACDS09} for the function $\phi=f^{[n]}$.
We use the properties of $T_\phi$ proven there, e.g. we use \cite[Lemma 4.3]{ACDS09}, to conclude that \eqref{def:Tfn} does not depend on the representation of $f^{[n]}$ as given by Lemma \ref{lem:divided diff}. We also note that \eqref{def:Tfn} in particular implies $T^D_{f^{[0]}}=f(D)$. 

An important property of the multiple operator integral $\Tfn^D$ is that it allows us to compute a higher-order derivative of a trace functional like the spectral action:
\begin{align}\label{eq:nth derivative of SA}
	\frac{1}{n!}\frac{d^n}{dt^n}\tr(f(D+tV))\big|_{t=0}=\tr(\Tfn^D(V,\ldots,V)),
\end{align}
see for instance \cite[Theorem 5.7]{ACDS09}. Under other assumptions, this formula is proven in \cite[Theorem 5.3.5]{ST19}, \cite[Theorem 3.12]{vNS21} and \cite[equation (5.30)]{PSS13}. In other words, multiple operator integrals are multilinear extensions of higher-order derivatives of the spectral action, which is why they are of interest to us. We write the Taylor series expansion of the spectral action as
\begin{align}\label{eq:Taylor}
	\tr(f(D+V))\sim\sum_{n=0}^\infty \tr(\Tfn^D(V,\ldots,V)).
\end{align}
  The remainder of this Taylor series can be written in terms of a single multiple operator integral:
  \begin{prop}
    \label{prop:Taylor}
For any $K \in\N_0$, let $f \in C^{K+1}$ be such that $\widehat{f^{(n)}} \in L^1$ for $n\in\{0,\ldots,K+1\}$. Then 
\begin{align}\label{eq:remainder}
  f(D+V)-f(D)-\sum_{n=1}^{K}\frac{1}{n!}\frac{d^n}{dt^n}f(D+tV)\Big|_{t=0}=T^{D+V,D,\ldots,D}_{f^{[K+1]}}(V,\ldots,V).
\end{align}

  \end{prop}
  \proof
The proof of \cite[Theorem 3.13]{vNS21} exactly goes through for the stated class of functions.
 \endproof
The following lemma is an adaptation of \cite[Lemma 4.6]{ACDS09}. We give its proof for convenience of the reader.

\begin{lem}\label{lem:T(V1,..,Vp) Schatten bound}
	Fix $n\in\N_0$ and $f\in C^n$ such that $\widehat{f^{(n)}}\in L^1$, and let $D_0,\ldots,D_n$ be self-adjoint in $\H$. Suppose $\alpha,\alpha_j\in[1,\infty]$ with $\frac{1}{\alpha}=\frac{1}{\alpha_1}+\ldots+\frac{1}{\alpha_n}$, and let $V_j\in\S^{\alpha_j}$, where we set $\S^\infty=B(\H)$. We then have
		$$\nrm{\Tfn^{D_0,\ldots,D_n}(V_1,\ldots,V_n)}{\alpha}\leq \frac{1}{n!}\absnorm{\widehat{f^{(n)}}}\nrm{V_1}{\alpha_1}\cdots\nrm{V_n}{\alpha_n}.$$
\end{lem}
\begin{proof}
	Assume that $\alpha<\infty$. Define
		$$A_{s,t}:=e^{its_0D_0}V_1e^{its_1D_1}\cdots V_ne^{its_nD_n},$$
	for all $(s,t)\in\Delta_n\times\R= \Sigma$. Using H\"older's inequality 
	we find that $A_{s,t}\in\S^\alpha$ and
		$$\nrm{A_{s,t}}{\alpha}\leq\nrm{V_1}{\alpha_1}\cdots\nrm{V_n}{\alpha_n}.$$
	Let $B\in\S^{\alpha'}$ be arbitrary, for $\alpha'$ the H\"older conjugate exponent of $\alpha$. Define
		$$g(s,t):=A_{s,t}B,$$
	and notice that $g:\Sigma\to\S^1$ is uniformly $\S^1$-bounded. For all $j\leq n$, the map $(s,t)\mapsto e^{it s_jD_j}$ is strongly continuous, and therefore $g:\Sigma\to\S^1$ is so*-continuous. 
	By \cite[Lemma 3.10]{ACDS09}, we obtain
	\begin{align*}
		\left|\Tr(\Tfn^{D_0,\ldots,D_n}(V_1,\ldots,V_n)B)\right|&=\left|\int_{\Sigma}\Tr(A_{s,t}B)\:d\sigma_f(s,t)\right|\\
		&\leq \int_{\Sigma}\Tr|A_{s,t}B|\:d|\sigma_f|(s,t)\\
		&\leq \nrm{V_1}{\alpha_1}\cdots\nrm{V_n}{\alpha_n}\nrm{B}{\alpha'}\frac{1}{n!}\supnorm{\widehat{f^{(n)}}}.
	\end{align*}
	When $\alpha=1$, simply taking a unitary $B$ such that $\Tfn^{D_0,\ldots,D_n}(V_1,\ldots,V_n)B=|\Tfn^{D_0,\ldots,D_n}(V_1,\ldots,V_n)|$ yields the result. If $\alpha\in(1,\infty)$, then an application of $(\S^{\alpha'})^*=\S^{\alpha}$ 
	gives the result. If $\alpha=\infty$, one obtains the lemma directly by the triangle inequality $\norm{\int A_{s,t}}\leq\int \norm{A_{s,t}}$.
\end{proof}

Lemma \ref{lem:T(V1,..,Vp) Schatten bound} shows in particular that $T^{D_0,\ldots,D_n}_{f^{[n]}}:B(\H)^{\times n}\to B(\H)$ is $\norm{\cdot}$-continuous. When we replace $(B(\H),\norm{\cdot})$ by $(B(\H)_1,\text{s.o.t.})$, this is known to still hold true (see \cite[Proposition 4.9]{ACDS09}). These results can be unified and generalized by writing $\L^\alpha:=(\S^\alpha,\nrm{\cdot}{\alpha})$ for $\alpha\in[1,\infty)$ and $\L^\infty:=(B(\H)_1,\text{s.o.t.})$. We then have the following.
\begin{lem}\label{lem:continuity with alpha relation}
	Let $f\in C^n$ with $\widehat{f^{(n)}}\in L^1$ and let $\alpha,\alpha_j\in[1,\infty]$ with $\tfrac{1}{\alpha}=\tfrac{1}{\alpha_1}+\ldots+\tfrac{1}{\alpha_n}$. Assume that either $\alpha_n<\infty$ or $\alpha_1=\ldots=\alpha_n=\infty$. Then
		$$T^{D_0,\ldots,D_n}_{f^{[n]}}:\L^{\alpha_1}\times\cdots\times\L^{\alpha_n}\to\L^\alpha$$
	is continuous.
\end{lem}
\begin{proof}
	If $\alpha=\infty$, then all $\alpha_j=\infty$, and the result is proven in \cite[Proposition 4.9]{ACDS09}. If $\alpha<\infty$, we can use the same technique and notation as in the proof of Lemma \ref{lem:T(V1,..,Vp) Schatten bound} to find
\begin{align}\label{eq:Tr(TfnC)=int Tr(AC)}
	\Tr(\Tfn^{D_0,\ldots,D_n}(V_1,\ldots,V_n)B)=\int_\Sigma \Tr(A_{s,t}B)d\sigma_f(s,t),
\end{align}
	for every $B\in\S^{\alpha'}$. By H\"older's inequality and the continuity of the product $\mB(\H)_1\times \L^p\to\L^p$,
	\begin{align}\label{eq:product}
		\L^p\times\L^q\to\L^r,\qquad (A,B)\mapsto AB
	\end{align}
	is continuous for every $p,q,r\in[1,\infty]$ with $\tfrac{1}{p}+\tfrac{1}{q}=\tfrac{1}{r}$ and either $q<\infty$ or $p=q=\infty$. Therefore also the product $\L^p\times (\L^q\times\L^s)\to\L^r$ is continuous whenever $p,q,r,s\in[1,\infty]$ satisfy $\tfrac{1}{p}+\tfrac{1}{q}+\tfrac{1}{s}=\tfrac{1}{r}$ and either $s<\infty$ or $p=q=s=\infty$. By repeatedly applying this continuity property, we find that the product
	\begin{align}\label{eq:product2}
		\L^{\alpha_1}\times\cdots\times\L^{\alpha_n}\to\L^\alpha,\qquad (V_1,\ldots,V_n)\mapsto V_1\cdots V_n
	\end{align}
	is continuous under the standing assumptions on $\alpha,\alpha_j$.
	Hence, for every $(s,t)\in\Sigma$, the operator $A_{s,t}B\in\L^1$ depends continuously on $(V_1,\ldots,V_n)\in\L^{\alpha_1}\times\cdots\times\L^{\alpha_n}$. Since a convergent sequence in $\L^p$ is bounded w.r.t. $\nrm{\cdot}{p}$ (where $\nrm{\cdot}{\infty}=\norm{\cdot}$) an application of the dominated convergence theorem shows that \eqref{eq:Tr(TfnC)=int Tr(AC)} depends sequentially continuously on $(V_1,\ldots,V_n)$. By our specific choice of $\L^\infty$, every $\L^p$ is a metric space, hence sequential continuity implies continuity.
\end{proof}

The relation $\tfrac{1}{\alpha}=\tfrac{1}{\alpha_1}+\ldots+\tfrac{1}{\alpha_n}$ is central to the above Lemma. When the resolvent of $D$ is $s$-Schatten, however, for an explicit class $\Wsn$ of functions $f$, defined in Section \ref{sct:Function spaces}, we can do away with this relation. We will prove this in \textsection\ref{sct:MOI continuity}.

\section{Estimating the multiple operator integral in the $s$-summable case}
\label{sct:Function spaces}
%

We specialize the class of functions $f$ that appear in Proposition \ref{prop:Taylor} and consider for $s,n\in\N_0$:
\begin{align*}
	\Wsn:=\{f\in C^{n}:~ \widehat{(fu^m)^{(k)}}\in L^1\text{ for all $m\leq s$ and $k\leq n$}\},
\end{align*}
where $u(x):=x-i$. Examples of functions in $\Wsn$ are $n+1$-differentiable functions such that $(fu^s)^{(k)}\in L^2$ for all $k\leq n+1$, such as Schwartz functions, or functions in $C_c^{n+1}$.

The main result of this section will be the following bound on the multiple operator integral.
\begin{thm}\label{thm:Schatten estimate}
	Let $D$ be self-adjoint in $\H$ such that $(D-i)^{-1}\in\S^s$ for $s\in\N$. For every $n\in\N_0$, every $f\in\Wsn$ and every $V_1,\ldots,V_n\in B(\H)$, the multiple operator integral $T^D_{f^{[n]}}(V_1,\ldots,V_n)$ is trace-class and satisfies the bound
		$$\nrm{T^D_{f^{[n]}}(V_1,\ldots,V_n)}{1}\leq c_{s,n}(f)\norm{V_1}\cdots\norm{V_n}\big\|(D-i)^{-1}\big\|_{s}^s,$$
	where
		$$c_{s,n}(f):=\sum_{k=0}^{\min(s,n)}\vect{s}{k}\frac{\absnorm{\widehat{(fu^{s-k})^{(n-k)}}}}{(n-k)!}.$$
	More generally, when $V\in\mB(\H)_\sa$,
		$$\nrm{T^{D+V,D,\ldots,D}_{f^{[n]}}(V_1,\ldots,V_n)}{1}\leq c_{s,n}(f)\norm{V_1}\cdots\norm{V_n}(1+\norm{V})^{2s}\big\|(D-i)^{-1}\big\|_{s}^s.$$
\end{thm}
This theorem allows us to freely work with the traces of multiple operator integrals up to order $n$. This is the sole analytical ingredient for a truncated version of our main result, namely Theorem \ref{thm:asymptotic expansion}. However, if we want the expansion \eqref{eq:expansion intro} to converge, we will need to impose infinite differentiability of $f$, as well as a growth condition on the $L^1$-norms occurring in Theorem \ref{thm:Schatten estimate}, as $n$ goes to infinity.
We therefore introduce the space
$$\Esg:=\left\{f\in C^\infty :~ \text{there exists $C_f\geq 1$ s.t. } \|\widehat{(fu^m)^{(n)}}\|_1\leq C_f^{n+1}n!^\gamma \,\text{ for all $m\leq s$ and $n\in\mathbb N_0$}\right\},$$
for $\gamma\in (0,1]$. Our main result is that the expansion \eqref{eq:expansion intro} holds for all functions $f\in\Esg$, and certain perturbations $A$. If $\gamma=1$, the expansion converges absolutely whenever the perturbation $A$ is sufficiently small. If $\gamma<1$ the expansion converges absolutely for all perturbations. The following Lemma underlies both results.
\begin{lem}\label{lem:Cs and Es bounds}
	Let $s\in\N$, $D$ self-adjoint in $\H$ with $(D-i)^{-1}\in\S^s$, and $\gamma\in(0,1]$. For any $f\in\Esg$ there exists a $C\geq1$ such that for all $n\in\N_0$, $V_1,\ldots,V_n\in\mB(\H)$, and $V\in\mB(\H)_\sa$, we have
		$$\absnorm{\Tfn^{D+V,D,\ldots,D}(V_1,\ldots,V_n)}\leq \Big(C^{n+1}n!^{\gamma-1}\Big)\norm{V_1}\cdots\norm{V_n}(1+\norm{V})^{2s}\big\|(D-i)^{-1}\big\|_s^s.$$
\end{lem}
\begin{proof}
Apply the definition of $\Esg$ to Theorem \ref{thm:Schatten estimate}, and absorb $2^s$ into the constant $C$.
\end{proof}

This lemma will be used in Subsection \ref{sct:convergence} and Section \ref{sct:vanishing pairing}.

Examples of functions in $\E^{s,1}$ are Schwartz functions with compactly supported Fourier transform. The following proposition gives more examples.

\begin{prop}
	Let $f\in C^\infty$ and $s,t\in\N_0$.
	\begin{enumerate}[label=(\roman*)]
		\item\label{mult E} If $f\in \E^{s,1}$ and $g\in \E^{t,1}$, then $fg\in \E^{s+t,1}$.
		\item\label{exp1 E} If $\hat{f}\in L^1$ with $|\hat{f}(x)|\leq e^{-c|x|}$ a.e. for some $c>0$, then $f\in\E^{0,1}$.
		\item\label{exp2 E} If $\widehat{fu^s}\in L^1$ with $|\widehat{fu^s}(x)|\leq e^{-c|x|}$ a.e. for some $c>0$, then $f\in\E^{s,1}$. 
		\item\label{rati E} Rational functions in $\O(|x|^{-s-1})$ are in $\E^{s,1}$.
		\item\label{Gauss E} The function $x\mapsto e^{-x^2}$ is in $\E^{s,1/2}$ for any $s\in\N_0$.
	\end{enumerate}
\end{prop}
\begin{proof}
	\begin{enumerate}[label=(\roman*)]
	\item For $m\leq s$ and $p\leq t$, Young's inequality gives
	\begin{align*}
		\|\widehat{(fgu^{m+p})^{(n)}}\|_1&\leq\sum_{k=0}^n\vect{n}{k}\|\widehat{(fu^m)^{(k)}}\|_{1}\|\widehat{(gu^p)^{(n-k)}}\|_{1}\\
		&\leq (n+1)n!(C_fC_g)^{n+1}.
	\end{align*}
	Any polynomial in $n$ is $\mathcal{O}(C^n)$ for some $C\geq1$.\\
	\item As
		$$\frac{1}{n!}(\tfrac{1}{2}c)^n|x|^n\leq\sum_{m=0}^\infty \frac{1}{m!}(\tfrac12c|x|)^m=e^{\tfrac12c|x|},$$
	we find $\|\widehat{f^{(n)}}\|_1=\||x|^n\hat{f}\|_1\leq \|e^{\tfrac12c|x|}\hat{f}\|_1(\tfrac12c)^{-n}n!$, thereby obtaining $f\in \E^{0,1}$.
	\item 
	 Item \ref{exp1 E} gives that $fu^s\in\E^{0,1}$. It is easy to see that $u^{-s}\in\E^{s-1,1}$. Therefore \ref{mult E} gives $f\in\E^{s-1,1}$, i.e., $\|\widehat{(fu^m)^{(n)}}\|_1\leq C_f^{n+1}n!$ for $m\leq s-1$. Similar to \ref{exp1 E} we get $\|\widehat{(fu^s)^{(n)}}\|_1\leq C^{n+1}n!$ for some $C\geq 1$.
	\item Follows from \ref{exp2 E}.
	
	\item Let $f(x)=e^{-x^2}$ and $m\in\N_0$. The Fourier transform of $fu^m$ is a polynomial times a Gaussian, say $(fu^m)\hat{~}(x)=p(x)e^{-x^2/c^2}$. Therefore,
\begin{align*}
	\absnorm{\widehat{(fu^m)^{(n)}}}=\absnorm{|x|^n p(x)e^{-x^2/c^2}}.
\end{align*}
Furthermore, $|x|^n=c^n\sqrt{(x^2/c^2)^{n}}\leq c^n\sqrt{n!e^{x^2/c^2}}=\sqrt{n!}c^ne^{\frac{x^2}{2c^2}}$, so
\begin{align*}
	\absnorm{\widehat{(fu^m)^{(n)}}}&\leq \sqrt{n!}c^n\absnorm{p(x)e^{-\frac{x^2}{2c^2}}}.
\end{align*}
Therefore, $f\in\E^{s,1/2}$ for any $s\in\N_0$.\qedhere
	\end{enumerate}
\end{proof}

\subsection{Bound on the multiple operator integral}\label{sct:core results}
We will now give a proof of Theorem \ref{thm:Schatten estimate}, which uses the summability of $D$ to obtain a trace class estimate on the multiple operator integral $\Tfn^D$. For summability $s=2$, a similar estimate was found by Anna Skripka in \cite[Lemma 3.6]{Skr13}. This was later generalized to the case of so-called relative Schatten perturbations in \cite[Theorem 3.10]{vNS21}, but only holds for $k$-multilinear operator integrals of order $k\geq s$.

The core idea of our proof is the same as of \cite[Lemma 3.6]{Skr13} and \cite[Theorem 3.10]{vNS21}, namely to expand $T_{f^{[n]}}(V_1,\ldots,V_n)$ as a sum of operator integrals, which should then be bounded using the triangle inequality and (a noncommutative) H\"older's inequality. 
For brevity, we sometimes write $V_k^{\{j\}}:=V_k(D_k-i)^{-j}$, and $T^{D_0,\ldots,D_k}_\phi(V_1,\ldots,V_k)^{\{j\}}:=T^{D_0,\ldots,D_k}_\phi(V_1,\ldots,V_k)(D_k-i)^{-j}$.

\begin{lem}\label{lem:added weights}
	For $s,n\in\N_0$, $f\in\Wsn$, $V_1,\ldots,V_n\in\mB(\H)$, and $D_0,\ldots,D_n$ self-adjoint in $\H$, we have
	\begin{align*}
		 T^{D_0,\ldots,D_n}_{f^{[n]}}(V_1,\ldots,V_n)=&
		\sum_{k=0}^{\min(s,n)}(-1)^k\!\!\!\!\!\!\!\sum_{\substack{ j_0\geq0,\, j_1,\ldots,j_k\geq 1,\\ j_0+\ldots+j_k=s}}
		 T^{D_0,\ldots,D_{n-k}}_{(fu^{s-k})^{[n-k]}} (V_1,\ldots,V_{n-k})^{\{j_0\}}V_{n-k+1}^{\{j_1\}}\cdots V_n^{\{j_k\}}.
	\end{align*}
\end{lem}
\begin{proof}
	We prove the lemma by induction on $s$. If $s=0$, the statement follows easily. 
	For the induction step, we note that the proof of \cite[Theorem 3.10(i)]{vNS21} implies that, for all $f\in\Wsn$, we have
		$$T^{D_0,\ldots,D_n}_{f^{[n]}}(V_1,\ldots,V_n)=T^{D_0,\ldots,D_n}_{(fu)^{[n]}}(V_1,\ldots,V_n)(D_n-i)^{-1}-T^{D_0,\ldots,D_{n-1}}_{f^{[n-1]}}(V_1,\ldots,V_{n-1})V_n(D_n-i)^{-1}.$$
	Suppose the claim of the lemma holds for a certain $s\in\N_0$. Then
	\begin{align*}
		&T^{D_0,\ldots,D_n}_{f^{[n]}}(V_1,\ldots,V_n)\\
		&\quad=\sum_{k=0}^{\min(s,n)}\sum_{\substack{j_0\geq0,\, j_1,\ldots,j_k\geq 1\\ j_0+\ldots+j_k=s}}  (-1)^kT^{D_0,\ldots,D_{n-k}}_{(fu^{s-k+1})^{[n-k]}}(V_1,\ldots,V_{n-k})^{\{j_0+1\}}V_{n-k+1}^{\{j_1\}}\cdots V_n^{\{j_k\}}\\
		&\qquad+\sum_{k=0}^{\min(s,n-1)}\sum_{\substack{j_0\geq0,\, j_1,\ldots,j_k\geq 1\\ j_0+\ldots+j_k=s}}(-1)^{k+1}T^{D_0,\ldots,D_{n-k-1}}_{(fu^{s-k})^{[n-k-1]}}(V_1,\ldots,V_{n-k-1})V_{n-k}^{\{j_0+1\}}V_{n-k+1}^{\{j_1\}}\cdots V_n^{\{j_k\}}\\
		&\quad=\sum_{k=0}^{\min(s,n)}\sum_{\substack{j_0\geq1,\,j_1,\ldots,j_k\geq1\\j_0+\ldots+j_k=s+1}}(-1)^k T^{D_0,\ldots,D_{n-k}}_{(fu^{s+1-k})^{[n-k]}}(V_1,\ldots,V_{n-k})^{\{j_0\}}V_{n_k+1}^{\{j_1\}}\cdots V_n^{\{j_k\}}\\
		&\qquad+\sum_{k=1}^{\min(s+1,n)}\sum_{\substack{j_0=0,\,j_1,\ldots,j_k\geq1\\ j_0+\ldots+j_k=s+1}}(-1)^kT^{D_0,\ldots,D_{n-k}}_{(fu^{s+1-k})^{[n-k]}}(V_1,\ldots,V_{n-k})^{\{j_0\}}V_{n-k+1}^{\{j_1\}}\cdots V_n^{\{j_k\}}.
	\end{align*}
	In the first term we can freely replace the sum from $k=0$ to $\min(s,n)$ by a sum from $k=0$ to $\min(s+1,n)$, because the appearing sum over $j_0,\ldots,j_k$ is trivial for $k=s+1$. Similarly, in the second term, we can freely let $k$ run from 0 to $\min(s+1,n)$. Combining the two terms gives the claim of the lemma for $s+1$, which completes the induction step.
\end{proof}

We prove the main result of this section. 

\begin{proof}[Proof of Theorem \ref{thm:Schatten estimate}.]
	We apply Lemma \ref{lem:added weights}, and find
	\begin{align*}
		\nrm{T^D_{f^{[n]}}(V_1,\ldots,V_n)}{1}\leq\sum_{k=0}^{\min(s,n)}\sum_{\substack{ j_0\geq0,\, j_1,\ldots,j_k\geq 1,\\ j_0+\ldots+j_k=s}}
		\nrm{ T^D_{(fu^{s-k})^{[n-k]}} (V_1,\ldots,V_{n-k})^{\{j_0\}}}{s \over j_0} \nrm{V_{n-k+1}^{\{j_1\}}}{s \over j_1}\cdots \nrm{V_n^{\{j_k\}}}{s \over j_k}\!\!.
	\end{align*}
	Apply Lemma \ref{lem:T(V1,..,Vp) Schatten bound}, to find
	\begin{align*}
		\nrm{T^D_{f^{[n]}}(V_1,\ldots,V_n)}{1}\leq\sum_{k=0}^{\min(s,n)}\sum_{\substack{ j_0\geq0,\, j_1,\ldots,j_k\geq 1,\\ j_0+\ldots+j_k=s}}\frac{1}{(n-k)!}\nrm{\widehat{(fu^{s-k})^{(n-k)}}}{1}\norm{V_1}\cdots\norm{V_n}\|(D-i)^{-1}\|_{s}^s.
	\end{align*}
	A bit of combinatorics shows that the sum over $j_0,\ldots,j_k$ adds a factor $\vect{s}{k}$, which implies the first statement of the theorem. The second statement follows similarly, with the added remark that
		$$\|(D+V-i)^{-1}\|_{s}^s\leq(1+\norm{V})^{2s}\|(D-i)^{-1}\|_{s}^s.$$
	See, e.g., \cite[Appendix B, Lemma 6]{CP}.
\end{proof}

\subsection{Continuity of the multiple operator integral}
\label{sct:MOI continuity}

\begin{thm}\label{thm:continuity for L's}
Let $s\in\N$, $D$ self-adjoint in $\H$ with $(D-i)^{-1}\in\S^s$, $n\in\N_0$, and $f\in\Wsn$. The map $$T^D_{f^{[n]}}:\L^\infty\times\cdots\times\L^\infty\to\L^1$$ is continuous. (Recall that $\L^\infty=B(\H)_1$, endowed with the strong operator topology.)
\end{thm}
\begin{proof}
	Suppose that 
	$V^m_1\to V_1,\ldots,V^m_n\to V_n$ in $\L^\infty$.
	By continuity of \eqref{eq:product}, we obtain that 
		$$(V^m_{n-k+l})^{\{j_l\}}\to V_{n-k+l}^{\{j_l\}}\quad\text{in $\L^{s/j_l}$.}$$
	We invoke Lemma \ref{lem:continuity with alpha relation} to find that
		$$T^D_{(fu^{s-k})^{[n-k]}}(V^m_1,\ldots,V^m_{n-k})^{\{j_0\}}\to T^D_{(fu^{s-k})^{[n-k]}}(V_1,\ldots,V_{n-k})^{\{j_0\}}\quad\text{in $\L^{s/j_0}$.}$$
	By Lemma \ref{lem:added weights} and the continuity of the product \eqref{eq:product2}  (for which we remark that the assumptions on $\alpha_j,\alpha$ are indeed satisfied for all terms obtained from Lemma \ref{lem:added weights}) we find that
		$$T^D_{f^{[n]}}(V^m_1,\ldots,V^m_n)\to T^D_{f^{[n]}}(V_1,\ldots,V_n)\quad\text{in $\L^{1}$},$$
	so we are done.
\end{proof}

To emphasize the strength of this result, we compare it to Lemma \ref{lem:continuity with alpha relation} which is (at least in the separate cases $\alpha_1,\ldots,\alpha_n<\infty$ and $\alpha=\infty$) known in the literature. Using the continuity of the inclusion $\L^\alpha\hookrightarrow\L^\infty$ we obtain the following clear improvement.
\begin{cor}
	Let $s\in\N$, $D$ self-adjoint in $\H$ with $(D-i)^{-1}\in\S^s$, $n\in\N_0$, and $f\in\Wsn$. For any $\alpha_1,\ldots,\alpha_n\in[1,\infty]$ and any $\alpha\in[1,\infty]$ (no relation between $\alpha$ and the $\alpha_j$'s assumed) the map
		$$\Tfn^D:\L^{\alpha_1}\times\cdots\times\L^{\alpha_n}\to\L^\alpha$$
	is continuous.
\end{cor}

\section{Cyclic cocycles and universal forms underlying the spectral action}
\label{sct:CC and UF underlying the SA}
Mainly to fix our conventions, we start with the definition of a finitely summable spectral triple, which is the situation in which our main result is stated.
\begin{defn}
  \label{defn:st}
	Let $s\in\N$. An \textbf{$s$-summable spectral triple} $(\A,\H,D)$ consists of a separable Hilbert space $\H$, a self-adjoint operator $D$ in $\H$ and a unital *-algebra $\A\subseteq B(\H)$, such that, for all $a\in\A$, $a\operatorname{dom} D\subseteq\operatorname{dom} D$ and $[D,a]$ extends to a bounded operator, and $(D-i)^{-1}\in\mathcal{S}^s$.
\end{defn}


Throughout this section, we let $(\A,\H,D)$ be an $s$-summable spectral triple for $s\in\N$, and we let $f\in\W^{s,n}$ for $n\in\N_0$, unless stated otherwise.

\begin{defn}\label{def:bracket}
	Define a multilinear function $\br{\cdot}:B(\H)^{\times n}\to\C$ by
	\begin{align}\label{eq:br cycl}
	\br{V_1,\ldots,V_n}:=\sum_{j=1}^n\tr( T^D_{f^{[n]}}(V_j,\ldots,V_n,V_1,\ldots,V_{j-1})).
	\end{align}
\end{defn}

For our algebraic results (which make up most of Section \ref{sct:CC and UF underlying the SA} and \textsection\ref{sct:truncated expansion}) we only need two simple properties of the bracket $\br{\cdot}$, stated in the following lemma. After proving this lemma, all analytical subtleties (related to the unboundedness of $D$) are taken care of, and we can focus on the algebra that ensues from these simple rules.
\begin{lem}\label{cycl bracket}
For $V_1,\ldots,V_n\in\mB(\H)$ and $a\in\A$ we have
\begin{enumerate}[label=(\Roman*)]
	\item $\br{V_1,\ldots,V_n}=\br{V_n,V_1,\ldots,V_{n-1}},$\label{cyclicity}
	\item $\br{V_1,\ldots,aV_j,\ldots,V_n}-\br{V_1,\ldots,V_{j-1}a,\ldots,V_n}=\br{V_1,\ldots,V_{j-1},[D,a],V_j,\ldots,V_n}$,\label{commutation}
\end{enumerate}
          where it is understood that for the edge case $j=1$ we need to substitute $n$ for $j-1$ on the left-hand side, and $f\in\W^{s,n+1}$ is assumed to define the right-hand side. 
\end{lem}
\begin{proof}
	Property \ref{cyclicity} follows immediately from Definition \ref{def:bracket}. By writing out the definitions for rank-1 operators $V_1,\ldots,V_n$, we have,
	\begin{align}
		&T^D_{f^{[n]}}(V_1,\ldots,V_j,aV_{j+1},\ldots,V_n)-T^D_{f^{[n]}}(V_1,\ldots,V_ja,V_{j+1},\ldots,V_n)\nonumber\\
		&\quad=T^D_{f^{[n+1]}}(V_1,\ldots,V_j,[D,a],V_{j+1},\ldots,V_n),\label{com1}
	\end{align}
	and the two edge cases,
	\begin{align}\label{com2}
		T^D_{f^{[n]}}(aV_1,\ldots,V_n)-aT^D_{f^{[n]}}(V_1,\ldots,V_n)&=T^D_{f^{[n+1]}}([D,a],V_1,\ldots,V_n),\\
		T^D_{f^{[n]}}(V_1,\ldots,V_n)a-T^D_{f^{[n]}}(V_1,\ldots,V_na)&=T^D_{f^{[n+1]}}(V_1,\ldots,V_n,[D,a]).\label{com3}
	\end{align}
	By Theorem \ref{thm:continuity for L's}, and the fact that the finite-rank operators lie strongly dense in $B(\H)$, we find that formulas \eqref{com1} , \eqref{com2} and \eqref{com3}
        hold for all $V_1,\ldots,V_n\in B(\H)$. Hence,
	\begin{align*}
		&\br{aV_1,V_2,\ldots,V_n}-\br{V_1,V_2,\ldots,V_na}\\
		&\quad = \sum_{j=2}^{n}\tr (T^D_{f^{[n]}}(V_j,\ldots,V_n,[D,a],V_1,\ldots,V_{j-1}))\\
		&\qquad +\tr(T^D_{f^{[n]}}(aV_1,\ldots,V_n))-\tr(T^D_{f^{[n]}}(V_1,\ldots,V_na))\\
		&\quad=	\sum_{j=2}^{n}\tr (T^D_{f^{[n+1]}}(V_{j},\ldots,V_n,[D,a],V_1,\ldots,V_{j-1}))\\
		&\qquad +\tr(T^D_{f^{[n+1]}}([D,a],V_1,\ldots,V_n))+\tr(aT^D_{f^{[n]}}(V_1,\ldots,V_n))-\tr(T^D_{f^{[n]}}(V_1,\ldots,V_na))\\
		&\quad=	\sum_{j=2}^{n}\tr (T^D_{f^{[n+1]}}(V_{j},\ldots,V_n,[D,a],V_1,\ldots,V_{j-1}))\\
		&\qquad +\tr(T^D_{f^{[n+1]}}([D,a],V_1,\ldots,V_n))+\tr(T^D_{f^{[n+1]}}(V_1,\ldots,V_n,[D,a]))\\
		&\quad= \br{[D,a],V_1,\ldots,V_n},
	\end{align*}
	and therefore \ref{commutation} follows by applying \ref{cyclicity}.
\end{proof}

\begin{rem}
Under additional assumptions -- for instance when $V_1,\ldots, V_n\in\S^1$ and $f\in\W^{s,n}$ is such that $f'$ is compactly supported and analytic in a region of $\C$ containing a rectifiable curve $\gamma$ which surrounds the support of $f$ in $\R$ -- we have
	\begin{align*}
		\br{V_1,\ldots,V_n}= \frac{1}{2\pi i}\tr\oint_\gamma f'(z)\prod_{j=1}^n V_j(z-D)^{-1}.
	\end{align*}
This occurs in \cite[Corl. 20]{Sui11} in the case where $V_1=V_2=\cdots=V_n$. It would be interesting to confront these resolvent formulas with the ones appearing in the work of Paycha \cite{Pay07}.
\end{rem}

\subsection{Hochschild and cyclic cocycles}
When the above brackets $\br{\cdot}$ are evaluated at one-forms $a[D,b]$ associated to a spectral triple, the relations found in Lemma \ref{cycl bracket} can be translated nicely in terms of the coboundary operators appearing in cyclic cohomology. This is very similar to the structure appearing in the context of index theory, see for instance \cite{GS89,Hig06}.

Let us start by recalling the definition of the boundary operators $b$ and $B$ from \cite{C85}.

\begin{defn}
If $\A$ is an algebra, and $n\in\N_0$, we define the space of {\em Hochschild $n$-cochains}, denoted by $\mathcal{C}^n(\A)$, as the space of $(n+1)$-linear
functionals $\phi$ on $\A$ with the property that if $a_j =1$ for some $j \geq 1$, then $\phi(a_0,\ldots,a_n) = 0$. Define operators $b : \mathcal{C}^{n}(\A) \to \mathcal{C}^{n+1}(\A)$ and $B: \mathcal{C}^{n+1}(\A) \to \mathcal{C}^{n}(\A)$ by
\begin{align*}
b\phi(a_0, a_1,\dots, a_{n+1})
:=& \sum_{j=0}^n (-1)^j \phi(a_0,\dots, a_j a_{j+1},\dots, a_{n+1})\\
& + (-1)^{n+1} \phi(a_{n+1} a_0, a_1,\dots, a_n) ,\\
B \phi(a_0 ,a_1, \ldots, a_n) :=& 
\sum_{j=0}^n (-1)^{nj}\phi(1,a_j,a_{j+1},\ldots, a_{j-1}).
\end{align*}
\end{defn}
Note that $B = A B_0$ in terms of the operator $A$ of cyclic anti-symmetrization and the operator defined by $B_0 \phi (a_0, a_1, \ldots, a_n) = \phi(1,a_0, a_1,\ldots, a_n)$.

One may check that the pair $(b,B)$ defines a double complex, {\em i.e.} $b^2 = 0, B^2=0$ and $bB +Bb =0$. Hochschild cohomology then arises as the cohomology of the complex  $(\mathcal{C}^n(\A),b)$, while the for us relevant {\em periodic cyclic cohomology} is defined as the cohomology of the totalization of the $(b,B)$-complex. That is to say, 
\begin{align*}
\mathcal{C}^\ev(\A) = \bigoplus_k \mathcal{C}^{2k} (\A) ; \qquad \mathcal{C}^{\odd}(\A) = \bigoplus_k \mathcal{C}^{2k+1} (\A),
\end{align*}
form a complex with differential $b+B$ and the cohomology of this complex is called periodic cyclic cohomology. We will also refer to a periodic cyclic cocycle as a $(b,B)$-cocycle. Explicitly, an odd $(b,B)$-cocycle is thus given by a sequence
$$
(\phi_1, \phi_3, \phi_5, \ldots),
$$
where $\phi_{2k+1} \in \mathcal{C}^{2k+1}(\A)$ and 
$$
b \phi_{2k+1} + B \phi_{2k+3} = 0 ,
$$
for all $k \geq 0$, and also $B \phi_1 = 0$. An analogous statement holds for even $(b,B)$-cocycles.

\subsection{Cyclic cocycles associated to multiple operator integrals}

We define the following Hochschild $n$-cochain:
\begin{align}\label{eq:def phi_n}
	\phi_n(a_0,\ldots,a_n):=\br{a_0[D,a_1],[D,a_2],\ldots,[D,a_{n}]} \qquad  (a_0, \ldots, a_n \in \A).
\end{align}
	We easily see that $B_0\phi_n$ is invariant under cyclic permutations, so that $B\phi_n=nB_0\phi_n$ for odd $n$ and $B\phi_n=0$ for even $n$. Also, $\phi_n(a_0,\ldots,a_n)=0$ when $a_j=1$ for some $j\geq1$. We put $\phi_0:=0$.
\begin{lem}\label{lem:b}
	We have $b\phi_n=\phi_{n+1}$ for odd $n$ and we have $b\phi_n=0$ for even $n$.
\end{lem}
\begin{proof}
	As $b\phi_0=0$ by definition, and $b^2=0$, we need only check the case in which $n$ is odd.
	
	We find, by splitting up the sum, and shifting the second appearing sum by one, that
	\begin{align*}
		&b\phi_n(a_0,\ldots,a_{n+1})\\
		&\quad=\br{a_0a_1[D,a_1],[D,a_2],\ldots,[D,a_{n+1}]}
		-\br{a_0a_1[D,a_1],[D,a_2],\ldots,[D,a_{n+1}]}\\
		&\qquad+\sum_{j=2}^n(-1)^j\br{a_0[D,a_1],[D,a_2],\ldots,a_j[D,a_{j+1}],\ldots,[D,a_{n+1}]}\\
		&\qquad-\sum_{j=2}^{n+1}(-1)^j\br{a_0[D,a_1],[D,a_2],\ldots,[D,a_{j-1}]a_j,\ldots,[D,a_{n+1}]}\\
		&\qquad+\br{a_{n+1}a_0[D,a_1],[D,a_2],\ldots,[D,a_n]}\\
		&\quad=\sum_{j=2}^n(-1)^j\br{a_0[D,a_1],[D,a_2],\ldots,[D,a_{n+1}]}
		-\br{a_0[D,a_1],[D,a_2],\ldots,[D,a_n]a_{n+1}}\\
		&\qquad+\br{a_{n+1}a_0[D,a_1],\ldots,[D,a_n]}\\
		&\quad=\br{[D,a_{n+1}],a_0[D,a_1],[D,a_2],\ldots,[D,a_{n}]}\\
		&\quad=\phi_{n+1}(a_0,\ldots,a_{n+1}),
	\end{align*}
	by \ref{cyclicity} and \ref{commutation} of Lemma \ref{cycl bracket}.
\end{proof}
\begin{lem}
Let $n$ be even. We have $bB_0\phi_n=2\phi_n-B_0\phi_{n+1}$.
\end{lem}
\begin{proof}
	Splitting the sum in two, and shifting the index of the second sum, we find
	\begin{align*}
		bB_0\phi_n(a_0,\ldots,a_n)&=\sum_{j=0}^{n-1}(-1)^j\br{[D,a_0],\ldots,a_j[D,a_{j+1}],\ldots,[D,a_n]}\\
		&\quad-\sum_{j=1}^n(-1)^j\br{[D,a_0],\ldots,[D,a_{j-1}]a_j,\ldots,[D,a_n]}+\br{[D,a_na_0],\ldots,[D,a_{n-1}]}\\
		&=\br{a_0[D,a_1],[D,a_2],\ldots,[D,a_n]}+\sum_{j=1}^{n-1}(-1)^j\br{[D,a_0],\ldots,[D,a_n]}\\
		&\quad-\br{[D,a_0],\ldots,[D,a_{n-2}],[D,a_{n-1}]a_n}+\br{[D,a_na_0],\ldots,[D,a_{n-1}]}\\
		&=\phi_n(a_0,\ldots,a_n)-\br{[D,a_0],\ldots,[D,a_n]}+\br{[D,a_n],[D,a_0],\ldots,[D,a_{n-1}]}\\
		&\quad+\br{[D,a_n]a_0,[D,a_1],\ldots,[D,a_{n-1}]}\\
		&=2\phi_n(a_0,\ldots,a_n)-B_0\phi_{n+1}(a_0,\ldots,a_n),
	\end{align*}
	by using both properties of the bracket $\br{\cdot}$ in the last step.
\end{proof}

Motivated by this we define 
	$$\psi_{2k-1}:=\phi_{2k-1}-\tfrac{1}{2}B_0\phi_{2k},$$
so that
$$B\psi_{2k+1}=2(2k+1)b\psi_{2k-1}.$$
We can rephrase this property in terms of the $(b,B)$-complex as follows. 
\begin{prop}
  \label{prop:bB}
  Let $\phi_n$ and $\psi_{2k-1}$ be as defined above and set
  $\tilde{\psi}_{2k-1}:=(-1)^{k-1}\frac{(k-1)!}{(2k-1)!}\psi_{2k-1}$.
  \begin{enumerate}
  \item The sequence $(\phi_{2k})$ is a $(b,B)$-cocycle and each $\phi_{2k}$ defines an even Hochschild cocycle: $b \phi_{2k} = 0$. 
    \item The sequence $(\tilde \psi_{2k-1})$ is an odd $(b,B)$-cocycle. 
    \end{enumerate}
  \end{prop}
We use an integral notation that is defined by linear extension of
	$$\int_\phi a_0da_1\cdots da_n :=\int_{\phi_n}a_0da_1\cdots da_n:=  \phi(a_0,a_1,\ldots,a_n),$$
and similarly for $\psi$.

\subsection{Derivatives of the spectral action in terms of universal forms}\label{sct:from V to A}
In this section we will express the derivatives of the fluctuated spectral action (occurring in the Taylor series) in terms of universal forms that are integrated along $\phi$. We thus make the jump from an expression in terms of $V=\pi_D(A)\in\Omega^1_D(\A)_\sa$ to an expression in terms of $A\in\Omega^1(\A)$. By \eqref{eq:nth derivative of SA} and the definition of $\br{V,\ldots,V}$, we have, for $n\in\N$,
\begin{align}
	\frac{1}{n!}\frac{d^n}{dt^n}\tr(f(D+tV))\big|_{t=0}&=\tr(\Tfn^D(V,\ldots,V))\nonumber\\
	&=\frac{1}{n}\br{V,\ldots,V}.\label{eq:<V,...,V>/n}
\end{align}
As $V$ decomposes as a finite sum $V=\sum a_j[D,b_j]$, our task is to express $\br{a_{j_1}[D,b_{j_1}],\ldots,a_{j_n}[D,b_{j_n}]}$ in terms of universal forms $a_0da_1\cdots da_n$ integrated along $\phi$. This will turn out to be possible by just using \ref{commutation} and $[D,a_1a_2]=a_1[D,a_2]+[D,a_1]a_2$. To find the exact expression we need to work in the algebra $M_2(\Omega^\bullet(\A))=M_2(\C)\otimes\Omega^\bullet(\A)$. 

\begin{prop}
	Let $n\in\N$. For $a_1,\ldots,a_n,b_1,\ldots,b_n\in\A$, denoting $A_j:=a_jdb_j$, we have
		$$\br{a_1[D,b_1],\ldots,a_n[D,b_n]}=\int_\phi \begin{pmatrix}
A_1&0
\end{pmatrix}\prod_{j=2}^n\begin{pmatrix}
A_j+dA_j&-A_j\\dA_j&-A_j
\end{pmatrix}\begin{pmatrix}
1\\0
\end{pmatrix}.$$
\end{prop}
\begin{proof}
 If we combine, for every $n\in\N_0$, the $n$-multilinear function $\br{\cdot}$ from \eqref{eq:br cycl}, we obtain a linear function
	$$\br{\cdot}:T\mB(\H)\to\C$$
	on the tensor algebra $T\mB(\H)$. 
	For any $\omega,\nu\in T\mB(\H)$, a straightforward calculation using the commutation rule \ref{commutation} from Lemma \ref{cycl bracket} shows that
\begin{align}
	&\br{\omega\otimes a_{j-1}[D,b_{j-1}]\otimes\begin{pmatrix}
	a_j&a_jb_j
	\end{pmatrix}\nu}
	=\br{\omega\otimes\begin{pmatrix}
	a_{j-1}&a_{j-1}b_{j-1}
	\end{pmatrix}
	M_j\otimes\nu},\label{eq:look at the stars}
\end{align}
where $M_j\in M_2(T\mB(\H))$ is defined by
\begin{align}\label{eq:M_j}
	M_j:=\begin{pmatrix}
	[D,b_{j-1}a_j]+[D,b_{j-1}]\otimes[D,a_j]&[D,b_{j-1}a_jb_j]+[D,b_{j-1}]\otimes[D,a_jb_j]\\-[D,a_j]&-[D,a_jb_j]
	\end{pmatrix}.
\end{align}
Repeating \eqref{eq:look at the stars}, and subsequently using \eqref{eq:def phi_n}, it follows that
\begin{align*}
	\br{a_1[D,b_1],\ldots, a_n[D,b_n]}&=\br{a_1[D,b_1]\otimes\ldots\otimes a_{n-1}[D,b_{n-1}]\otimes\begin{pmatrix}a_n& a_nb_n\end{pmatrix}\begin{pmatrix}
	[D,b_n]\\0
	\end{pmatrix}}\\
	&=\br{\begin{pmatrix}
	a_1&a_1b_1
	\end{pmatrix}\bigg(\prod_{j=2}^nM_j\bigg)\begin{pmatrix}
	[D,b_n]\\0
	\end{pmatrix}}\\
	&=\int_\phi \begin{pmatrix}
	a_1&a_1b_1
	\end{pmatrix}\bigg(\prod_{j=2}^nN_j\bigg)\begin{pmatrix}
	db_n\\0
	\end{pmatrix},
\end{align*}
where from \eqref{eq:M_j} we obtain
\begin{align*}	
	N_j&=\begin{pmatrix}
	d(b_{j-1}a_j)+db_{j-1} da_j&d(b_{j-1}a_jb_j)+db_{j-1} d(a_jb_j)\\-da_j&-d(a_jb_j)
	\end{pmatrix}\\
	&=\begin{pmatrix}
	db_{j-1}&b_{j-1}\\
	0&-1
	\end{pmatrix}
	\begin{pmatrix}
	a_j+da_j&a_jb_j+da_jb_j+a_jdb_j\\
	da_j&da_jb_j+a_jdb_j
	\end{pmatrix}.
\end{align*}
By also writing $\begin{pmatrix}
	db_n\\0
	\end{pmatrix}=\begin{pmatrix}
	db_n&b_n\\0&-1
	\end{pmatrix}\begin{pmatrix}
	1\\0
	\end{pmatrix}$, we find that
\begin{align*}
	&\br{a_1[D,b_1],\ldots a_n[D,b_n]}\\
	&\quad=\int_\phi \begin{pmatrix}
	a_1&a_1b_1
	\end{pmatrix}\begin{pmatrix}
	db_{1}&b_{1}\\
	0&-1
	\end{pmatrix}\left(\prod_{j=2}^n\begin{pmatrix}
	a_j+da_j&a_jb_j+da_jb_j+a_jdb_j\\
	da_j&da_jb_j+a_jdb_j
	\end{pmatrix}\begin{pmatrix}
	db_{j}&b_{j}\\
	0&-1
	\end{pmatrix}\right)\begin{pmatrix}
	1\\0
	\end{pmatrix}\\
	&\quad=\int_\phi \begin{pmatrix}
	A_1&0
	\end{pmatrix}\left(\prod_{j=2}^n\begin{pmatrix}
	A_j+dA_j&-A_j\\
	dA_j&-A_j
	\end{pmatrix}\right)\begin{pmatrix}
	1\\0
	\end{pmatrix},
\end{align*}
	which concludes the proof.
\end{proof}

\begin{cor}\label{cor:2x2 matrix}
	If $n\in\N$, $A\in\Omega^1(\A)$ and $V:=\pi_D(A)\in\Omega^1_D(\A)$, then
\begin{align}\label{eq:2x2 matrix}
	\br{V,\ldots,V}=\int_\phi \begin{pmatrix}
	A & 0
	\end{pmatrix}
	\begin{pmatrix}
	A+dA & -A\\
	dA   & -A
	\end{pmatrix}^{n-1}\begin{pmatrix}
	1\\
	0
	\end{pmatrix}.
\end{align}
\end{cor}
	
\begin{ex}\label{ex:first terms}
	Using \eqref{eq:2x2 matrix}, we obtain in particular
\begin{align*}
	\br{V}&=\int_{\phi_1}A,\\
	\br{V,V}&=\int_{\phi_2}A^2+\int_{\phi_3}AdA,\\
	\br{V,V,V}&=\int_{\phi_3}A^3+\int_{\phi_4}AdAA+\int_{\phi_5}AdAdA,\\
	\br{V,V,V,V}&=\int_{\phi_4}A^4+\int_{\phi_5}(A^3dA+AdAA^2)+\int_{\phi_6}AdAdAA+\int_{\phi_7}AdAdAdA.
\end{align*}
	With \eqref{eq:<V,...,V>/n}, and in the sense of \eqref{eq:Taylor}, this implies that
	\begin{align*}
		\tr(f(D+V)-f(D))=\int_{\phi_1} A+\frac{1}{2}\int_{\phi_2}A^2+\int_{\phi_3}\Big(\frac{1}{2}AdA+\frac{1}{3}A^3\Big)+\int_{\phi_4}\Big(\frac{1}{3}AdAA+\frac{1}{4}A^4\Big)+\ldots,
	\end{align*}
	where the dots indicate terms of degree 5 and higher. Using $\phi_{2k-1}=\psi_{2k-1}+\frac{1}{2}B_0\phi_{2k}$, this becomes
	\begin{align*}
		\tr(f(D+V)-f(D))=&\int_{\psi_1} A+\frac{1}{2}\int_{\phi_2}(A^2+dA)+\int_{\psi_3}\Big(\frac{1}{2}AdA+\frac{1}{3}A^3\Big)\\
		&+\frac{1}{4}\int_{\phi_4}\Big(dAdA+\frac{2}{3}(dAA^2+AdAA+A^2dA)+A^4\Big)+\ldots.
	\end{align*}
	Notice that, if $\phi_4$ would be tracial, we would be able to identify the terms $dAA^2$, $AdAA$ and $A^2dA$, and thus obtain the Yang--Mills form $F^2=(dA+A^2)^2$, under the fourth integral. In the general case, however, cyclic permutations under $\int_\phi$ produce correction terms, of which we will need to keep track.
\end{ex}

\subsection{Near-tracial behavior of $\int_\phi$}\label{sct:Cyclicity}

In \textsection\ref{sct:from V to A}, we have not used the cyclicity property \ref{cyclicity} from Lemma \ref{cycl bracket}. Doing so yields the following proposition, which shows how $\int_\phi$ differs from being tracial. This proposition and its corollary are crucial for Section \ref{sct:main thm}.

For a universal $n$-form $X\in\Omega^n(\A)$, define $\odd(X):=1$ if $n$ is odd, and $\odd(X):=0$ if $n$ is even.
\begin{prop}\label{prop:cycl}
	Let $X$ and $Y$ be universal forms. Then
		$$\int_\phi XY-\int_\phi YX=\odd(Y)\int_\phi Y dX-\odd(X)\int_\phi X dY.$$
\end{prop}
\begin{proof}
Without loss of generality, assume that $X=x_0dx_1\dots dx_n$ and $Y=y_0dy_1\dots dy_k$ for some $x_0,\ldots,x_n,y_0,\ldots,y_k\in\A$.
By using $da b=d(ab)-adb$ repeatedly, we get
\begin{align*}
	\int_\phi XY =& \int_\phi x_0 dx_1\cdots dx_{n-1}(d(x_ny_0)-x_ndy_0)dy_1\cdots dy_k\\
	=&\int_\phi x_0\big( dx_1\cdots dx_{n-1}d(x_ny_0)-dx_1\cdots dx_{n-2}d(x_{n-1}x_n)dy_0+\dots\\
	&\ldots+(-1)^{n-1}d(x_1x_2)dx_3\cdots dx_ndy_0+(-1)^nx_1dx_2\cdots dx_ndy_0\big)dy_1\cdots dy_k\\
	=&\big\langle x_0[D,x_1],\ldots,[D,x_{n-1}],[D,x_ny_0],[D,y_1],\ldots,[D,y_k]\big\rangle\\
	&-\big\langle x_0[D,x_1],\ldots,[D,x_{n-2}],[D,x_{n-1}x_n],[D,y_0],[D,y_1],\ldots,[D,y_k]\big\rangle+\ldots\\
	&\ldots+(-1)^{n-1}\big\langle x_0[D,x_1x_2],[D,x_3],\ldots,[D,x_n],[D,y_0],[D,y_1],\ldots,[D,y_k]\big\rangle\\
	&+(-1)^n\big\langle x_0x_1[D,x_2],[D,x_3],\ldots,[D,x_n],[D,y_0],[D,y_1],\ldots,[D,y_k]\big\rangle\\
	=&\big\langle x_0[D,x_1],\ldots,[D,x_{n}]y_0,[D,y_1],\ldots,[D,y_k]\big\rangle\\
	&+\sum_{j=0}^{n-2}(-1)^j\big\langle x_0[D,x_1],\ldots,[D,x_n],[D,y_0],\ldots,[D,y_k]\big\rangle\\
	=&\big\langle x_0[D,x_1],\ldots,[D,x_{n}],y_0[D,y_1],\ldots,[D,y_k]\big\rangle-\odd(X)\int_\phi X dY.
\end{align*}
	Doing the same for $\int_\phi YX$ and using cyclicity (Lemma \ref{cycl bracket}\ref{cyclicity}) yields the proposition.
\end{proof}

A quick check shows that the above proposition implies the following handy rules.

\begin{cor}\label{cor:cycl}Let $X,Y\in\Omega^\bullet(\A)$, and $A\in\Omega^1(\A)$.
\begin{enumerate}[label=(\roman*)]
	\item\label{cycl2} If $X$ and $Y$ are both of even degree, then
		$$\int_\phi XY=\int_\phi YX.$$
	\item\label{cycl3} If $X$ has odd degree, then
		$$\int_\phi (AX-XA)=\int_\phi d(AX).$$
	\item\label{cycl4} If $X$ has even degree, then
		$$\int_\phi(XA-AX)=\int_\phi dXA+\int_\phi dAdX.$$
\end{enumerate}	

\end{cor}


\subsection{Higher-order generalized Chern--Simons forms}
As a final preparation for our main result, we briefly recall from \cite{Qui90} the definition of Chern--Simons forms. 

         \begin{defn}
           \label{defn:cs}
           Let $(\Omega^\bullet, d)$ be a differential graded algebra. The {\em Chern--Simons form} of degree $2k-1$ is given for $A \in \Omega^1$ by 
           \begin{equation}\label{eq:cs}
\cs_{2k-1}(A) := \int_0^1 A (F_t)^{k-1} dt,
           \end{equation}
           where $F_t = t dA + t^ 2 A^2$ is the curvature two-form of the (connection) one-form $A_t = t A$.
           \end{defn}
         We will only work with the universal differential graded algebra $\Omega^\bullet=\Omega^\bullet(\A)$ for the algebra $\A$.

         \begin{ex}
           For the first three Chern--Simons forms one easily derives the following explicit expressions:
           \begin{gather*}
             \cs_1(A) = A; \qquad  \cs_3(A) = \frac 12 \left( A dA + \frac 2 3 A^3 \right);\\
             \cs_5(A) = \frac 13 \left( A (dA)^2 + \frac 3 4 A dA A^ 2 + \frac 3 4 A^3 dA + \frac 3 5 A^5 \right).
             \end{gather*}
These are well-known expressions from the physics literature, {\em cf.} \cite[Section 11.5.2]{Nak90}.
           \end{ex}

\section{Expansion of the spectral action in terms of $(b,B)$-cocycles}\label{sct:main thm}
In this section we prove our main theorem.
\begin{thm}\label{thm:main thm}
  Let $(\A,\H,D)$ be an $s$-summable spectral triple, and let $f\in\E^{s,\gamma}$ for $\gamma\in(0,1)$. The spectral action fluctuated by $V=\pi_D(A)\in\Omega_D^1(\A)_\sa$ 
  can be written as
 \begin{align*}
  \tr(f(D+V)-f(D)) = \sum_{k=1}^\infty \left( \int_{\psi_{2k-1}}  \cs_{2k-1} (A) +\frac 1 {2k} \int_{\phi_{2k}}  F^{k} \right),
  \end{align*}
  where the series converges absolutely.
\end{thm}
We prove this theorem in two steps. Firstly, we deal with the algebraic part of this statement, in \textsection\ref{sct:truncated expansion}. Here we only need to assume $f\in\Wsn$ for a finite $n\in\N$. Secondly, in \textsection\ref{sct:convergence}, we tackle the analytical part. We there obtain a strong estimate on the remainder of the above expansion in Theorem \ref{thm:main bounds} for a function $f\in\E^{s,\gamma}$ for general $\gamma\in(0,1]$. This estimate will imply that the conclusion of Theorem \ref{thm:main thm} still holds in the case of $\gamma=1$, when the perturbation $V$ is sufficiently small. When $f\in\E^{s,\gamma}$ for $\gamma\in(0,1)$, the expansion follows for all perturbations, and thus we prove Theorem \ref{thm:main thm}.

\subsection{Truncated expansion}\label{sct:truncated expansion}
Let $K\in\N$, $f\in\W^{s,2K}$, and $V=\pi_D(A)\in\Omega^1_D(\A)_\sa$.
We prove a truncated version of Theorem \ref{thm:main thm}, showing that the fluctuation of the spectral action can be expressed in terms of Chern--Simons and Yang--Mills forms, up to a remainder which involves forms of degree higher than $K$.
To enumerate the remainder forms we use the index set
\begin{align}\label{eq:T_K}
  T_K:=\left\{(v,w,p)\in\coprod_{m\in\N_0} (\N_0\times\N^{m-1})\times\N^m\times\N_0~\middle|~|v|+|w|+\left\lfloor\frac{p}{2}\right\rfloor< K,~ 2|v|+|w|+p\geq K\right\}.
\end{align}
\begin{prop}\label{prop:k}We have the asymptotic expansion
\begin{align*}
	\tr(f(D+V)-f(D))\sim&\sum_{k=1}^\infty\int_\phi\bigg(\cs_{2k-1}(A)+\int_0^1AF_t^{k-1}tA\,dt\bigg),
\end{align*}
by which we mean that we can write the $K^\text{th}$ remainder of this expansion as
\begin{align*}
	&\tr(f(D+V)-f(D))-\sum_{k=1}^K\int_\phi\bigg(\cs_{2k-1}(A)+\int_0^1AF_t^{k-1}tA\,dt\bigg)\\
	&\quad= \tr\left(T_{f^{[K+1]}}^{D+V,D,\ldots,D}(V,\ldots,V)\right)-\sum_{(v,w,p)\in T_K}\frac{1}{2|v|+|w|+p+1}\int_\phi AA^{2v_1}(dA)^{w_1}\cdots A^{2v_m}(dA)^{w_m}A^{p},
\end{align*}
where $T_K$, defined by \eqref{eq:T_K}, satisfies $|T_K|\leq 2^{K+1}$, and where $f\in\W^{s,2K}$.
\end{prop}
\begin{proof}
We start with the 2x2 matrix equation from Corollary \ref{cor:2x2 matrix} and separate the 1-forms $A$ from the two-forms $dA$. The $n$-th term in the Taylor expansion of $\tr (f(D+V))$ is given (by use of \eqref{eq:<V,...,V>/n} and Corollary \ref{cor:2x2 matrix}) by
\begin{align}
	\frac{1}{n!}\frac{d^{n}}{dt^{n}}\tr(f(D+tV))\Big|_{t=0}
	=&
	\frac{1}{n}\int_\phi\begin{pmatrix}
	A & 0
	\end{pmatrix}
	\left(\begin{pmatrix}
	A & -A\\
	0   & -A
	\end{pmatrix}+
	\begin{pmatrix}
	dA & 0\\
	dA   & 0
	\end{pmatrix}
	\right)^{n-1}\begin{pmatrix}
	1\\
	0
	\end{pmatrix}\nonumber\\
	\equiv&
	\frac{1}{n}\int_\phi\begin{pmatrix}
	A & 0
	\end{pmatrix}
	\left(\alpha A+
	\beta dA\right)^{n-1}\begin{pmatrix}
	1\\
	0
	\end{pmatrix}\nonumber\\
	=& \frac{1}{n}\int_\phi A\,e_1^t (\alpha A+\beta dA)^{n-1}e_1.
	\label{2x2 formula Tr}
\end{align}
	for some scalar-valued 2x2 matrices $\alpha$ and $\beta$, and $e_1=\vect{1}{0}$. The $\alpha$'s and $\beta$'s have very nice algebraic properties, which can be used to regroup the terms in the expansion in $n$.
When summing \eqref{2x2 formula Tr} from $n=1$ to infinity, and grouping the universal forms by their degree as in Example \ref{ex:first terms}, we need some machinery to keep track of the coefficient $1/n$. We will work in the space of (finite) polynomials $M_2(\Omega^\bullet(\A))[t]$, and define an integration with respect to $t$ as the linear map $\int_0^1dt\,:M_2(\Omega^\bullet(\A))[t]\to M_2(\Omega^\bullet(\A))$ given by integration of polynomials. We thus obtain
\begin{align}
	\frac{1}{n!}\frac{d^{n}}{dt^{n}}\tr(f(D+tV))\Big|_{t=0}=& \int_0^1dt\,t^{n-1}\int_\phi A \,e_1^t(\alpha A+\beta dA)^{n-1}e_1\nonumber\\
	=&\int_0^1dt\int_\phi A\,e_1^t(\alpha tA+\beta t dA)^{n-1}e_1.\label{eq:M_2[t]}
\end{align}

We now expand the $(n-1)$-th power, which is complicated because $\alpha$ and $\beta$ do not commute. To avoid notational clutter, let us denote $X:=tA$ and $Y:=tdA$. We find
\begin{align}
	e_1^t (\alpha X+\beta Y)^{n-1}e_1=\sum_{k=0}^{\lceil\frac{n-1}{2}\rceil}\sum_{\substack{v_1\geq0,~v_2,\ldots,v_k\geq1\\ w_1,\ldots,w_k\geq 1,~p\geq0\\ |v|+|w|+p=n-1}}e_1^t (\alpha^{v_1}\beta^{w_1}\cdots\alpha^{v_k}\beta^{w_k}\alpha^{p})e_1\, X^{v_1}Y^{w_1}\cdots X^{v_k}Y^{w_k}X^{p}.\label{ncbinom}
\end{align}

We can summarize the identities involving $\alpha$ and $\beta$ that we will use as
\begin{align*}
	\alpha^2&=1;	&
	\beta^2&=\beta;	&
	\beta\alpha\beta&=0;& e_1^t(\alpha) e_1&=1;\\
	e_1^t (\alpha\beta\alpha) e_1&=0;	&	
	e_1^t (\alpha\beta) e_1 &=0;	&
	e_1^t (\beta\alpha) e_1 &= 1;	&
	e_1^t(\beta) e_1&=1.
\end{align*}
From these identities follow the following two remarks:
\begin{itemize}
\item If $k\geq2$ and $v_i$ is odd for a certain $i\in\{2,\ldots,k\}$, then somewhere in $\alpha^{v_1}\beta^{w_1}\cdots\alpha^{v_k}\beta^{w_k}\alpha^{p}$ a factor $\beta\alpha\beta=0$ occurs, so in particular $$e_1^t( \alpha^{v_1}\beta^{w_1}\cdots\alpha^{v_k}\beta^{w_k}\alpha^{p})e_1=0.$$
\item If $v_1$ is odd and $v_2,\ldots,v_k$ are all even, then
	$$e_1^t(\alpha^{v_1}\beta^{w_1}\cdots\alpha^{v_k}\beta^{w_k}\alpha^{p})e_1=e_1^t( \alpha\beta\alpha^{p})e_1=0.$$
\end{itemize}
Therefore, for all $k\geq 0$, we conclude that in \eqref{ncbinom} only terms remain in which $v_1,\ldots,v_k$ are even. In fact, we find
\begin{align*}
	e_1^t (\alpha X+\beta Y)^{n-1}e_1&=\sum_{k=0}^{\lceil\frac{n-1}{2}\rceil}\sum_{\substack{v_1\in2\N_0,~v_2,\ldots,v_k\in2\N 
	\\w_1,\ldots,w_k\geq 1,~p\geq0,\\|v|+|w|+p=n-1}}e_1^t( \alpha^{v_1}\beta^{w_1}\cdots\alpha^{v_k}\beta^{w_k}\alpha^{p})e_1\, X^{v_1}Y^{w_1}\cdots X^{v_k}Y^{w_k}X^{p}\\
	&=\sum_{k=0}^{\lceil{\frac{n-1}{2}}\rfloor}\sum_{\substack{v_1\in2\N_0,~v_2,\ldots,v_k\in2\N\\w_1,\ldots,w_k\geq 1,~p\geq0,\\|v|+|w|+p=n-1}}e_1^t( \beta\alpha^{p})e_1\, X^{v_1}Y^{w_1}\cdots X^{v_k}Y^{w_k}X^{p}\\
	&=\sum_{k=0}^{\lceil\frac{n-1}{2}\rceil}\sum_{\substack{v_1\geq0,~v_2\ldots,v_k\geq1,\\w_1,\ldots,w_k\geq 1,~ p\geq0,\\2|v|+|w|+p=n-1}} (X^2)^{v_1}Y^{w_1}\cdots (X^2)^{v_k}Y^{w_k}X^{p}.
\end{align*}
Summing this from $n=1$ to $K$, we can write
\begin{align}\label{eq:one expansion}
	\sum_{n=1}^K e_1^t (\alpha X+\beta Y)^{n-1}e_1 = \sum_{(v,w,p)\in P_K}(X^2)^{v_1}Y^{w_1}\cdots (X^2)^{v_m}Y^{w_m}X^{p}
	,
\end{align}
where $P_K$ is the set of $(v,w,p)\in\coprod_{m}(\N_0\times\N^{m-1})\times\N^m\times\N_0$ such that $2|v|+w+p< K$.
%
In this last expression we can almost recognize an expansion of $(X^2+Y)^{k-1}=F_t^{k-1}$. Indeed, we have
\begin{align}\label{eq:other expansion}
	\sum_{k=1}^K(X^2+Y)^{k-1}(1+X)=\sum_{(v,w,p)\in S_K} (X^2)^{v_1}Y^{w_1}\cdots(X^2)^{v_m}Y^{w_m}X^{p},
\end{align}
where $S_K$ is the set of $(v,w,p)\in\coprod_{m}(\N_0\times\N^{m-1})\times\N^m\times\N_0$ such that $|v|+|w|+\lfloor \tfrac{p}{2}\rfloor< K.$ By \eqref{eq:other expansion} we have $|S_K|\leq 2^{K+1}$.
By using $T_K=S_K\setminus P_K$, we can combine \eqref{eq:M_2[t]}, \eqref{eq:one expansion} and \eqref{eq:other expansion}, and obtain
\begin{align}
	&\sum_{n=1}^{K}\frac{1}{n!}\frac{d^n}{dt^n}\tr\big(f(D+tV)\big)\Big|_{t=0}-\sum_{k=1}^K\int_\phi\int_0^1AF_t^{k-1}(1+tA)\,dt\nonumber\\
	&\quad= -\sum_{(v,w,p)\in T_K}\frac{1}{2|v|+|w|+p+1}\int_\phi AA^{2v_1}(dA)^{w_1}\cdots A^{2v_m}(dA)^{w_m}A^{p}.\label{eq:difference of expansions}
\end{align}
Together with \eqref{eq:remainder}, and the definition \eqref{eq:cs} of $\cs_{2k-1}(A)$, \eqref{eq:difference of expansions} implies the proposition.
\end{proof}

\begin{thm}\label{thm:asymptotic expansion}
We have
	$$\int_\phi\bigg(\cs_{2k-1}(A)+\int_0^1AF_t^{k-1}tA\,dt\bigg)=\int_{\psi_{2k-1}}\cs_{2k-1}(A)+\frac{1}{2k}\int_{\phi_{2k}} F^{k}.$$
	Therefore, with the same remainder term as in Proposition \ref{prop:k}, we have
	$$\tr(f(D+V)-f(D))\sim\sum_{k=1}^\infty\bigg(\int_{\psi_{2k-1}}\cs_{2k-1}(A)+\frac{1}{2k}\int_{\phi_{2k}} F^{k}\bigg).$$
\end{thm}

The Chern--Simons term in Proposition \ref{prop:k}, integrated along $\phi$, yields the correct Chern--Simons term integrated along $\psi$, plus an additional term. Indeed, recall that $\phi_{2k-1} = \psi_{2k-1} + \frac 12 B_0 \phi_{2k}$ so that we find
\begin{align*}
  \int_{\phi_{2k-1}}  A F_t^{k-1} +  \int_{\phi_{2k-1}} t A F_t^{k-1} A &=  \int_{\psi_{2k-1}}  A F_t^{k-1} +  \int_{\phi_{2k}} \Big( \frac 12  d(A F_t^{k-1} ) + t A F_t^{k-1} A\Big)\\
  &=  \int_{\psi_{2k-1}}  A F_t^{k-1} + \frac 12  \int_{\phi_{2k}} ( d A F_t^{k-1} + t A^2 F_t^{k-1}+ t A F_t^{k-1} A),
\end{align*}
where we used the repeated Bianchi identity $d(F_t^{k-1}) = - [A_t,F_t^{k-1}]$ in going to the last line.

We arrive at the following formula:
$$
\tr(f(D+V)-f(D)) \sim \sum_{k=1}^\infty \left( \int_{\psi_{2k-1}}  \cs_{2k-1} (A)  + \frac 12 \int_0^1 dt \int_{\phi_{2k}} ( d A F_t^{k-1} + t A^2 F_t^{k-1}+ t A F_t^{k-1} A)\right).
$$
We are now to show that the second term, namely
\begin{align*}
	YM_k:=&\frac{1}{2}\int_0^1 dt\int_{\phi_{2k}}(dA F_t^{k-1}+tA^2F_t^{k-1}+tAF_t^{k-1}A)\\
	=&\int_0^1 dt\frac{1}{2t}\int_{\phi_{2k}}(dA_t F_t^{k-1}+A_t^2F_t^{k-1}+A_tF_t^{k-1}A_t),
\end{align*}
equals $\frac{1}{2k}\int_{\phi_{2k}} F^{k}$.
After some rearrangement we can use Corollary \ref{cor:cycl}\ref{cycl3}, to find
\begin{align}
	YM_k&=\int_0^1 dt\frac{1}{2t}\int_{\phi_{2k}}(dA_t+2A_t^2) F_t^{k-1}+\int_0^1 dt\frac{1}{2t}\int_{\phi_{2k}}\left(A_tF_t^{k-1}A_t-A_t^2F_t^{k-1}\right)\nonumber\\
	&=\int_0^1 dt\frac{1}{2t}\int_{\phi_{2k}}(dA_t+2A_t^2) F_t^{k-1} - \int_0^1 dt\frac{1}{2t}\int_{\phi_{2k+1}} d(A_t^2 F_t^{k-1}).\label{eq:M}
\end{align}
We will first show that the second term of \eqref{eq:M} vanishes. We use the following rule, which allows us to replace the integrand by a form which is two degrees lower.
\begin{lem}\label{lem:een}
	For every $m\geq0$, we have
		$$\int_{\phi_{2m+3}} d(A_t^2F_t^m)=-\int_{\phi_{2m+1}} \left(d(A_t^2F_t^{m-1})+dA_td(F_t^{m-1})\right).$$
\end{lem}
\begin{proof}
	We use the definition of $F_t$, the repeated Bianchi identity $d(F_t^m)=[F_t^m,A_t]$, and subsequently Proposition \ref{prop:cycl}, to obtain
	\begin{align*}
		\int_{\phi_{2m+1}} d(A_t^2F_t^{m-1})=&\int_{\phi_{2m+1}} \left(d(F_tF_t^{m-1})-d(dA_tF_t^{m-1})\right)\\
		=&\int_{\phi_{2m+1}} \left(d(F_t^m)-dA_td(F_t^{m-1})\right)\\
		=&\int_{\phi_{2m+1}}\left(F_t^mA_t-A_tF_t^m\right)-\int_{\phi_{2m+1}} dA_td(F_t^{m-1})\\
		=&\int_{\phi_{2m+2}} A_td(F_t^m)-\int_{\phi_{2m+1}} dA_t d(F_t^{m-1})\\
		=&\int_{\phi_{2m+2}} \left(A_tF_t^mA_t-A_t^2F_t^m\right)-\int_{\phi_{2m+1}} dA_td(F_t^{m-1}).
	\end{align*}
	Applying Corollary \ref{cor:cycl}\ref{cycl3} to the first term gives
	\begin{align*}
		\int_{\phi_{2m+1}} d(A_t^2F_t^{m-1})=& -\int_{\phi_{2m+3}} d(A_t^2F_t^m)-\int_{\phi_{2m+1}} dA_td(F_t^{m-1}),
	\end{align*}
	which implies the lemma.
\end{proof}
We obtain the following result.
\begin{lem}\label{lem:een b}
	For every $m\geq0$, we have $$\int_{\phi_{2m+3}} d(A_t^2F_t^m)=0.$$
\end{lem}
\begin{proof}
	This is easily checked when $m=0$, and when $m=1$, it follows from Lemma \ref{lem:een}. If $m\geq 2$, we apply Lemma \ref{lem:een} twice, and find
	\begin{align*}
		\int_{\phi_{2m+3}} d(A_t^2F_t^m)=&-\int_{\phi_{2m+1}} d(A_t^2F_t^{m-1})-\int_{\phi_{2m+1}} dA_td(F_t^{m-1})\\
		=&\int_{\phi_{2m-1}} d(A_t^2F_t^{m-2})+\int_{\phi_{2m-1}} dA_td(F_t^{m-2})-\int_{\phi_{2m+1}} dA_td(F_t^{m-1}).
	\end{align*}
	We recognize $F_t=A_t^2+dA_t$ in the first two terms above, so by the Bianchi identity we obtain
	\begin{align*}
		\int_{\phi_{2m+3}} d(A_t^2F_t^m)=&\int_{\phi_{2m-1}} d(F_t^{m-1})-\int_{\phi_{2m+1}} dA_td(F_t^{m-1})\\
		=&\int_{\phi_{2m-1}} (F_t^{m-1}A_t-A_tF_t^{m-1})-\int_{\phi_{2m+1}} dA_td(F_t^{m-1}).
	\end{align*}
	By Corollary \ref{cor:cycl}\ref{cycl4}, the Bianchi identity, and Corollary \ref{cor:cycl}\ref{cycl2}, this gives
	\begin{align*}
		\int_{\phi_{2m+3}} d(A_t^2F_t^m)=&\int_{\phi_{2m}} d(F_t^{m-1})A_t\\
		=&\int_{\phi_{2m}} (F_t^{m-1}A_t^2-A_tF_t^{m-1}A_t)\\
		=&\int_{\phi_{2m}} (A_t^2F_t^{m-1}-A_tF_t^{m-1}A_t).
	\end{align*}
	We apply Corollary \ref{cor:cycl}\ref{cycl3}, to find
	\begin{align*}
		\int_{\phi_{2m+3}} d(A_t^2F_t^m)=&\int_{\phi_{2m+1}} d(A_t^2F_t^{m-1}).
	\end{align*}
	By induction, it follows that $\int_\phi d(A_t^2F_t^m)=0$ for all $m$.
\end{proof}
By the above Lemma, only the first term of \eqref{eq:M} remains, namely,
\begin{align}\label{eq:last term of M}
	YM_k=&\int_0^1 dt\frac{1}{2t}\int_{\phi_{2k}}(dA_t+2A_t^2) F_t^{k-1}	=\int_0^1 dt\int_{\phi_{2k}}(\tfrac{1}{2}dA+tA^2) F_t^{k-1}.
\end{align}
To express $YM_k$ in an even simpler form, we now remove the integral over $t$, which is possible by the following lemma.
\begin{lem}\label{lem:twee}
	We have
		$$\int_0^1 dt\int_{\phi_{2k}} (\tfrac{1}{2}dA+tA^2)F_t^{k-1}=\frac{1}{2k}\int_{\phi_{2k}}F^{k}.$$
\end{lem}
\begin{proof}
	Recall that $\Omega^\bullet(\A)[t]$ is the space of polynomials with coefficients in the algebra $\Omega^\bullet(\A)$. The linear map $\frac{d}{dt}:\Omega^\bullet(\A)[t]\rightarrow\Omega^\bullet(\A)[t]$ is defined by $\frac{d}{dt}(t^nB):=nt^{n-1}B$ for $B\in\Omega^\bullet(\A)$, and satisfies the Leibniz rule. Therefore,
	$$\frac{d}{dt}(F_t^{k})=\frac{d}{dt}(F_t)F^{k-1}_t+F_t\frac{d}{dt}(F_t)F_t^{k-2}+\ldots+F_t^{k-1}\frac{d}{dt}(F_t).$$
	Both $F_t$ and $\frac{d}{dt}(F_t)$ are $2$-forms, so, after a few applications of Corollary \ref{cor:cycl}\ref{cycl2}, we arrive at
	\begin{align*}
		\int_{\phi_{2k}}\frac{d}{dt}(F_t^{k})&=k\int_{\phi_{2k}}\frac{d}{dt}(F_t)F_t^{k-1}=k\int_{\phi_{2k}}(dA+2tA^2)F_t^{k-1}.
	\end{align*}
	The fundamental theorem of calculus (for polynomials) gives
	\begin{align*}
		\int_0^1 dt\int_{\phi_{2k}}\,(dA+2tA^2)F_t^{k-1}=&\frac{1}{k}\int_{\phi_{2k}}\int_0^1dt \,\frac{d}{dt}(F_t^{k})	=\frac{1}{k}\int_{\phi_{2k}}(F_1^{k}-F_0^{k})=\frac{1}{k}\int_{\phi_{2k}}F^{k},
	\end{align*}
	from which the lemma follows.
\end{proof}

\begin{proof}[Proof of Theorem \ref{thm:asymptotic expansion}]
Applying Lemma \ref{lem:twee} to our earlier expression for $YM_k$ (equation \eqref{eq:last term of M}), we find that
\begin{align*}
	YM_k=\frac{1}{2k}\int_{\phi_{2k}}F^{k}.
\end{align*}
We therefore obtain the theorem.
\end{proof}

\subsection{Convergence}\label{sct:convergence}
\noindent We prove a strong bound on the asymptotic expansion given by Theorem \ref{thm:asymptotic expansion}, in particular giving sufficient conditions for the series to converge, effectively replacing $\sim$ by $=$. A crucial ingredient is Lemma \ref{lem:Cs and Es bounds}.

\begin{thm}\label{thm:main bounds}
	Let $(\A,\H,D)$ be an $s$-summable spectral triple, let $n\in\N$, and fix $f\in \Esg$ for $\gamma\in(0,1]$. Then there exists $C_{f,s,n,\gamma}$ such that, for $A=\sum_{j=1}^na_jdb_j$ and $V=\sum_{j=1}^n a_j[D,b_j]$ self-adjoint with $\|a_j\|,\|b_j\|,\|[D,a_j]\|,\|[D,b_j]\|\leq R$, we have
	\begin{align}
		&\left|\tr(f(D+V)-f(D))-\sum_{k=1}^K\left(\int_{\psi_{2k-1}}\cs_{2k-1}(A)+\frac{1}{2k}\int_{\phi_{2k}}F^{k}\right)\right|\nonumber\\
		&\quad\leq \frac{C_{f,s,n,\gamma}^{K+1}}{K!^{1-\gamma}}\max(R^{2K+2},R^{4K+2+4s})\tr |(D-i)^{-s}|,\label{eq:convergence}
	\end{align}
	for all $K\in\N_0$. Moreover, we have
		$$\left|\int_{\psi_{2k-1}}\cs_{2k-1}(A)\right|+\left|\int_{\phi_{2k}}F^{k}\right|\leq \frac{C_{f,s,n,\gamma}^{k}}{k!^{1-\gamma}}\max(R^{2k},R^{4k})\tr|(D-i)^{-s}|.$$
\end{thm}
\begin{proof}
Theorem \ref{thm:asymptotic expansion} gives
\begin{align}
	&\left|\tr(f(D+V)-f(D))-\sum_{k=1}^K \left(\int_{\psi_{2k-1}}\cs_{2k-1}(A)+\frac{1}{2k}\int_{\phi_{2k}}F^{k}\right)\right|\nonumber\\
	&\quad\leq \left\|T_{f^{[K+1]}}^{D+V,D,\ldots,D}(V,\ldots,V)\right\|_1+\sum_{(v,w,p)\in T_K}\left|\int_\phi AA^{2v_1}(dA)^{w_1}\cdots A^{2v_m}(dA)^{w_m}A^{p}\right|.\label{eq:propk}
\end{align}
We first focus on the first term. Lemma \ref{lem:Cs and Es bounds} gives a $C\geq1$ such that
\begin{align*}
	\left\|T_{f^{[K+1]}}^{D+V,D,\ldots,D}(V,\ldots,V)\right\|_1&\leq \frac{C^{K+1}}{K!^{1-\gamma}}
	\sum_{j_1,\ldots,j_{K+1}\in\{1,\ldots,n\}}\prod_{m=1}^{K+1}\norm{a_{j_m}}\norm{[D,b_{j_m}]}(1+\norm{V})^{2s}\tr|(D-i)^{-s}|\\
	&\leq \frac{C^{K+1}}{K!^{1-\gamma}}n^{K+1}R^{2K+2}(1+\norm{V})^{2s}\tr|(D-i)^{-s}|,
\end{align*}
for all $K\in\N_0$. We conclude that there exists $\tilde C_{f,s,n,\gamma}$ such that
\begin{align*}
	\left\|T_{f^{[K+2]}}^{D+V,D,\ldots,D}(V,\ldots,V)\right\|_1\leq \frac{\tilde C_{f,s,n,\gamma}^{K+1}}{K!^{1-\gamma}}\max(R^{2K+2},R^{2K+2+4s})\tr|(D-i)^{-s}|.
\end{align*}

We now move on to the second term (the finite sum) on the right-hand side of \eqref{eq:propk}. It contains terms of the form
\begin{align*}
	\left|\int_\phi B_1\cdots B_M\right|,
\end{align*}
for $B_1,\ldots,B_M\in \{a_jdb_j,da_jdb_j:j\in\{1,\ldots,n\}\}$. Let $l$ be the degree of $B_1\cdots B_M$. By the definition of $T_K$ (equation \eqref{eq:T_K}) $K+1\leq l\leq 2K+1$ and $K+1\leq M\leq 2K+1$. By using the Leibniz rule repeatedly, we can write
\begin{align*}
	\int_\phi B_1\cdots B_M=\sum_{j\in J}\int_\phi e_{0,j}de_{1,j}\cdots de_{l,j},
\end{align*}
for a set $J$ with $|J|\leq 3^M\leq 3^{2K+1}$, and $e_{i,j}\in\A$ such that $e_{0,j}\cdots e_{l,j}=\prod_{m=1}^M a_{j_m}b_{j_m}$.
We get
\begin{align}
	\left|\int_\phi B_1\cdots B_M\right|&\leq\sum_{j\in J}|\int_\phi e_{0,j}de_{1,j}\cdots de_{l,j}|\nonumber\\
	&=\sum_{j\in J}|\phi_l(e_{0,j},\ldots,e_{l,j})|\nonumber\\
	&\leq \sum_{j\in J}\sum_{i=1}^l|\tr(T^D_{f^{[l]}}([D,e_i],\ldots,[D,e_l],e_0[D,e_1],[D,e_2],\ldots,[D,e_{i-1}])|,\label{eq:phi to MOI}
\end{align}
where we suppressed the index $j$ for readability.
We now apply Lemma \ref{lem:Cs and Es bounds} with $V=0$ to \eqref{eq:phi to MOI} and obtain
\begin{align*}
	\left|\int_\phi B_1\cdots B_M\right|
	&\leq  lC^{l+1}l!^{\gamma-1}\sum_{j\in J}\norm{e_0}\Big(\prod_{i=1}^l\norm{[D,e_i]}\Big)\tr|(D-i)^{-s}|,
\end{align*}
for a constant $C\geq 1$. Because we have $\norm{a_j},\norm{b_j}$, $\norm{[D,a_j]},\norm{[D,b_j]}\leq R$ by assumption, and $e_{0}\cdots e_{l}=\prod_{m=1}^M a_{j_m}b_{j_m}$, with $K+1\leq M\leq 2K+1$, we find
\begin{align*}
	\left|\int_\phi B_1\cdots B_M\right|
	&\leq \tilde C^{l+1}l!^{\gamma-1}\sum_{j\in J} R^{2M}\tr|(D-i)^{-s}|\\
	&\leq \hat C^{K+1}|J|K!^{\gamma-1}\max(R^{2K+2},R^{4K+2})\tr|(D-i)^{-s}|\\
	&\leq \check C^{K+1}K!^{\gamma-1}\max(R^{2K+2},R^{4K+2})\tr|(D-i)^{-s}|.
\end{align*}
We can now bound the second term on the right-hand side of \eqref{eq:propk}. We use that $|T_K|\leq 2^{K+1}$, and that $n^M\leq(n^2)^{K+1}$,
to find
\begin{align*}
	&\sum_{(v,w,p)\in T_K}\left|\int_\phi AA^{2v_1}(dA)^{w_1}\cdots A^{2v_m}(dA)^{w_m}A^{p}\right|\\
	&\quad\leq 2^{K+1}(n^2)^{K+1}\check C^{K+1}K!^{\gamma-1}\max(R^{2K+2},R^{4K+2})\tr|(D-i)^{-s}|\\
	&\quad\leq \breve C_{f,s,n,\gamma}^{K+1}K!^{\gamma-1}\max(R^{2K+2},R^{4K+2})\tr|(D-i)^{-s}|.
\end{align*}
Combining the first and second term of \eqref{eq:propk}, we obtain a number $C_{f,s,n,\gamma}$ such that \eqref{eq:convergence} holds. 

Moving on to the last claim of the theorem, we notice that, because $\psi_{2k-1}=\phi_{2k-1}-\frac{1}{2}B_0\phi_{2k}$,
	$$\left|\int_{\psi_{2k-1}}\cs_{2k-1}(A)\right|\leq\sum_{j\in J}\left|\int_\phi e_{0,j}de_{1,j}\cdots de_{l_j,j}\right|,$$
	where the sum is over certain $e_{i,j}\in\mathcal{A}$ (because $\mathcal{A}$ is unital) with $e_{0,j}\cdots e_{l_j,j}=\prod_{m=1}^{M}a_{j_m}b_{j_m}$ for some $M$ with $k\leq M\leq 2k-1$. The number of elements in $J$ is exponential in $k$. We obtain
	$$\left|\int_{\psi_{2k-1}}\cs_{2k-1}(A)\right|\leq \check C_{f,s,n,\gamma}^{k}k!^{\gamma-1}\max(R^{2k},R^{4k-2})\tr|(D-i)^{-s}|,$$
	for some number $\check C_{f,s,n,\gamma}\geq1$. Similarly, we obtain a number $\hat C_{f,s,n,\gamma}\geq1$ such that
	$$\left|\int_{\phi_{2k}}F^{k}\right|\leq \hat C_{f,s,n,\gamma}^{k}k!^{\gamma-1}\max(R^{2k},R^{4k})\tr|(D-i)^{-s}|,$$
	thereby proving the theorem.
\end{proof}
This theorem has two important corollaries, for $f\in\E^{s,1}$ (hence, for all $f\in\E^{s,\gamma}$) and for $f\in\E^{s,\gamma}$, $\gamma<1$.

\begin{cor}
	Let $(\A,\H,D)$ be an $s$-summable spectral triple, let $f\in\E^{s,1}$ and $V=\pi_D(A)\in\Omega^1_D(\A)_\sa$. Then there exists a $\delta>0$ such that for all $t\in\R$ with $|t|<\delta$, we have
	$$\tr(f(D+tV)-f(D))=\sum_{k=1}^\infty\left(\int_{\psi_{2k-1}}\cs_{2k-1}(tA)+\frac{1}{2k}\int_{\phi_{2k}}F_t^{k}\right),$$
	and the series converges absolutely.
\end{cor}
\begin{proof}
	Write $V=\sum_{j=1}^na_j[D,b_j]$. First take $C_{f,s,n,1}\geq 1$ from Theorem \ref{thm:main bounds}, define $R:=1/(C_{f,s,n,1}+1)$ such that $C_{f,s,n,1}R<1$, and define $\delta:= \left(\frac{R}{\max_j\{\|a_j\|,\|b_j\|,\|[D,a_j]\|,\|[D,b_j]\|\}}\right)^2$. By writing
	$$tV=\sum_{j=1}^n\sqrt{|t|}a_j[D,\text{sign}(t)\sqrt{|t|}b_j],$$
	the corollary follows.
\end{proof}

\begin{proof}[Proof of Theorem \ref{thm:main thm}.]
	This follows from Theorem \ref{thm:main bounds} by taking $\gamma<1$.
\end{proof}

\section{Gauge invariance and the pairing with K-theory}\label{sct:vanishing pairing}
Since the spectral action is a spectral invariant, it is in particular invariant under conjugation of $D$ by a unitary in the algebra $\A$. More generally, in the presence of an inner fluctuation we find that the spectral action is invariant under the transformation
$$
D+V \mapsto u (D+V) u^* = D + V^u; \qquad V^u = u[D,u^*] + u V u^*.
$$
This transformation also holds at the level of the universal forms, with a gauge transformation of the form $A \mapsto A^u = u d u^* + u A u ^*$. Let us analyze the behavior of the Chern--Simons and Yang--Mills terms appearing in Theorem \ref{thm:main thm} under this gauge transformation, and derive an interesting consequence for the pairing between the odd $(b,B)$-cocycle $\tilde \psi$ with the odd K-theory group of $\A$.

\begin{lem}
  The Yang--Mills terms $\int_{\phi_{2k}} F^k$ with $F = dA +A^2$ are invariant under the gauge transformation $A \mapsto A^u$ for every $k \geq 1$. 
  \end{lem}
\proof
Since the curvature of $A^u$ is simply given by $u F u^*$, the claim follows 
from Corollary \ref{cor:cycl}\ref{cycl2}.
\endproof


We are thus led to the conclusion that the Chern--Simons forms are gauge invariant as well. Indeed, arguing as in \cite{CC06}, since both $\tr f(D+V)$ and the Yang-Mills terms are invariant under $V \mapsto V^u$, we find that, under the assumptions stated in Theorem \ref{thm:main thm}:
$$
\sum_{k=0}^\infty \int_{\psi_{2k+1}} \cs_{2k+1}(A^u ) =  
\sum_{k=0}^\infty \int_{\psi_{2k+1}} \cs_{2k+1} (A ).
$$
Each individual Chern--Simons form behaves non-trivially under a gauge transformation. Nevertheless, it turns out that we can conclude, just as in \cite{CC06}, that the pairing of the whole $(b,B)$-cocycle with K-theory is trivial. 
Since the $(b,B)$-cocycle $\tilde \psi$ is given as an infinite sequence, we should first carefully study the analytical behavior of $\tilde \psi$. In fact, we should show that it is an {\em entire cyclic cocycle} in the sense of \cite{C88a} (see also \cite[Section IV.7.$\alpha$]{C94}). For this purpose, we can without loss of generality assume that $\A$ is complete in the Banach algebra norm defined by $\|a\|_1:=\|a\|+\|[D,a]\|$, because $(\overline{\A}^{\|\cdot\|_1},\H,D)$ is also a spectral triple, and the resulting $\tilde\psi_{2k+1}\in \mathcal{C}^{2k+1}(\overline{\A}^{\|\cdot\|_1})$ is an extension of the one in $\mathcal{C}^{2k+1}(\A)$. Recall that for Banach algebras $\A$ an odd cochain such as $\tilde \psi$ is called entire if the power series $\sum_k  \frac{(2k+1)!}{k!}\| \tilde \psi_{2k+1} \| z^k$ converges everywhere in $\C$. This is equivalent \cite[Remark IV.7.7a,c]{C94} to the condition that for any bounded subset $\Sigma \subset \A$ there exists a constant $C_\Sigma$ so that
$$
\left|\tilde \psi_{2k+1} (a_0, \ldots, a_{2k+1}) \right|  \leq \frac{C_\Sigma}{ k!} \qquad (\forall a_j \in \Sigma).
$$
In our case it turns out that Lemma \ref{lem:Cs and Es bounds} implies the following growth condition, guaranteeing that indeed $\tilde\psi$ is entire.

\begin{lem}
	Fix $f\in\Esg$ for $\gamma<1$ and equip $\A$ with the norm $\| a \|_1 = \| a \| + \| [D,a]\|$. Then, for any bounded subset $\Sigma\subset\mathcal{A}$ there exists $C_\Sigma$ such that
		$$\left |\tilde{\psi}_{2k+1}(a_0,\ldots, a_{2k+1}) \right|\leq \frac{C_\Sigma}{k!},$$
	for all $a_j\in\Sigma$.
\end{lem}

\begin{proof}
	Assume that $\| a_j \|_1 \leq R$ for all $a_j \in \Sigma$ so that both $\|a_j\|,\|[D,a_j]\|\leq R$. By definition of $\phi$, the expression $\psi_{2k+1}(a_0,\ldots,a_{2k+1})$ is given by a linear combination of multiple operator integrals with arguments in $\{V\in\mB(\H):\|V\|\leq R\}$ except for $a_0[D,a_1]$, which is bounded by $R^2$. By applying Lemma \ref{lem:Cs and Es bounds}, we obtain the estimate
	\begin{align}\label{eq:bound psi}
	\left	|\psi_{2k+1}(a_0,\ldots,a_{2k+1}) \right|\leq\bigg((2k+1)\frac{C^{2k+2}}{(2k+1)!^{1-\gamma}}+(k+1)\frac{C^{2k+3}}{(2k+2)!^{1-\gamma}}\bigg)R^{2k+2}\|(D-i)^{-1}\|_{s}^s.
	\end{align}
	We recall from Proposition \ref{prop:bB} that
	\begin{align}\label{tilde psi reprise}
		\tilde{\psi}_{2k+1}=(-1)^{k}\frac{k!}{(2k+1)!}\psi_{2k+1},
	\end{align}
	so that \eqref{eq:bound psi} in particular implies the lemma by use of, for instance, Stirling's approximation.
\end{proof}
For $u\in M_q(\A)$, define a pairing
\begin{align}\label{eq:pairing}
	\langle u,\tilde\psi\rangle:=(2\pi i)^{-1/2}\sum_{k=0}^\infty (-1)^{k}k!\tilde{\psi}^q_{2k+1}(u^*,u,\ldots,u^*,u),
\end{align}
where $\tilde{\psi}^q_{2k+1}:=\tr \# \tilde\psi_{2k+1}:(\mu_0\otimes a_0,\ldots,\mu_{2k+1}\otimes a_{2k+1})\mapsto\tr(\mu_0\cdots\mu_{2k+1})\tilde\psi_{2k+1}(a_0,\ldots,a_{2k+1})$ for $\mu_0,\ldots,\mu_{2k+1}\in M_q(\C)$ and $a_0,\ldots,a_{2k+1}\in\A$. Since $\tilde \psi$ is a $(b,B)$-cocycle, it follows from \cite[Corollary IV.7.27]{C94} (see also \cite[Sections III.3 and IV.7]{C94}) that this pairing only depends on the class of $u$ in $K_1(\A)$.

\begin{thm}
  Let $f\in\Esg$ for $\gamma<1$. Then the pairing of the odd entire cyclic cocycle $\tilde \psi$ with $K_1(\A)$ is trivial, {\em i.e.}
  $$\langle u,\tilde\psi\rangle=0
  $$
  for all unitary $u\in M_q(\A)$.
\end{thm}
\begin{proof}
	Apply Theorem \ref{thm:main thm} to a bigger spectral triple, namely $(\A^q,\H^{q},D^q):=(M_q(\C)\otimes \A,\C^q\otimes\H,I_q\otimes D)$. Take $A=u^*du$ for $u$ unitary in $M_q(\A)=M_q(\C)\otimes\A$.
	Clearly, then $V=u^*[D^q,u]$, and because the multiple operator integral behaves naturally with respect to tensor products, we obtain
		$$\tr(f(D^q+u^*[D^q,u])-f(D^q))=\sum_{k=0}^\infty\left(\int_{\psi^q_{2k+1}}\cs_{2k+1}(u^*du)+\frac{1}{2k+2}\int_{\phi^q_{2k+2}}F^{k+1}\right),$$
	where $F=d(u^*du)+(u^*du)^2=0$. Also notice that the left-hand side equals $\tr(f(u^*D^qu)-f(D^q))=0.$
	Therefore,
	\begin{align}\label{eq:cs is 0}
		\sum_{k=0}^\infty\int_{\psi^q_{2k+1}}\cs_{2k+1}(u^*du)=0.
	\end{align}
        From the definition of the Chern--Simons form (Definition \ref{defn:cs}) and the fact that $F_t=tdA+t^2A^2=(t-t^2)dA+t^2F=(t-t^2)du^*du$ we find that
	\begin{align*}
		\cs_{2k+1}(u^*du)&=\int_0^1dt\, (t-t^2)^{k}u^*dudu^*du\cdots du^*du,
	\end{align*}
	so that by a straightforward integration we may conclude that
	\begin{align*}
		\int_{\psi^q_{2k+1}}\cs_{2k+1}(u^*du)=\frac{k!^2}{(2k+1)!}\psi^q_{2k+1}(u^*,u,\ldots,u^*,u).
	\end{align*}
Combining this with \eqref{tilde psi reprise}, \eqref{eq:pairing} and \eqref{eq:cs is 0}, the theorem follows.
\end{proof}

\paragraph{Acknowledgements} We would like to thank Steven Lord and Fedor Sukochev for their generous hospitality during a visit in Summer 2019. We also thank them, the members in their research group, as well as the other participants in the workshop "Noncommutative Calculus and the Spectral Action" at UNSW in August 2019 for fruitful discussions. We also thank Alain Connes and Anna Skripka for useful comments. Research supported by NWO Physics Projectruimte (680-91-101).

\newcommand{\noopsort}[1]{}\def\cprime{$'$}


\begin{thebibliography}{99}

\bibitem{ACDS09}
N.~A.~Azamov, A.~L.~Carey, P.~G.~Dodds, and F.~A.~Sukochev.
Operator integrals, spectral shift, and spectral flow.
{\it Canad. J. Math.} (2) \textbf{61} (2009), 241--263.

\bibitem{CC96}
A.~H. Chamseddine and A.~Connes.
\newblock Universal formula for noncommutative geometry actions: {U}nifications
  of gravity and the {S}tandard {M}odel.
\newblock {\em Phys. Rev. Lett.} \textbf{77} (1996),  4868--4871.

\bibitem{CC97}
A.~H. Chamseddine and A.~Connes.
\newblock The spectral action principle.
\newblock {\em Commun. Math. Phys.} \textbf{186} (1997),  731--750.

\bibitem{CCM07}
A.~H. Chamseddine, A.~Connes, and M.~Marcolli.
\newblock Gravity and the {S}tandard {M}odel with neutrino mixing.
\newblock {\em Adv. Theor. Math. Phys.} (6) \textbf{11} (2007)  991--1089.

\bibitem{CD89}
Y.~Choquet-Bruhat and C.~DeWitt-Morette.
\newblock {\em Analysis, manifolds and physics. {P}art {II}}.
\newblock North-Holland Publishing Co., Amsterdam, 1989.
\newblock 92 applications.

\bibitem{CDD77}
Y.~Choquet-Bruhat, C.~DeWitt-Morette, and M.~Dillard-Bleick.
\newblock {\em Analysis, manifolds and physics}.
\newblock North-Holland Publishing Co., Amsterdam-New York-Oxford, 1977.

\bibitem{C85}
A.~Connes.
\newblock Non-commutative differential geometry.
\newblock {\em Publ. Math. IHES} \textbf{62} (1985)  257--360.

\bibitem{C88a}
A.~Connes.
\newblock Entire cyclic cohomology of {B}anach algebras and characters of
  {$\theta$}-summable {F}redholm modules.
\newblock {\em $K$-Theory} \textbf{1} (1988)  519--548.

\bibitem{C94}
A.~Connes.
\newblock {\em Noncommutative Geometry}.
\newblock Academic Press, San Diego, 1994.

\bibitem{C96}
A.~Connes.
\newblock Gravity coupled with matter and the foundation of non-commutative
  geometry.
  \newblock {\em Commun. Math. Phys.} \textbf{182} (1996)  155--176.
  
\bibitem{CC06}
A.~Connes and A.~H. Chamseddine.
\newblock {Inner fluctuations of the spectral action}.
\newblock {\em J. Geom. Phys.} \textbf{57} (2006) 1--21.

\bibitem{CS18}
A.~ Chattopadhyay and A.~ Skripka. 
\newblock {Trace formulas for relative Schatten class perturbations}.
\newblock {\em J. Funct. Anal.} (12) \textbf{274} (2018), 3377--3410.

\bibitem{CM07}
A.~Connes and M.~Marcolli.
\newblock {\em Noncommutative Geometry, Quantum Fields and Motives}.
\newblock AMS, Providence, 2008.

\bibitem{CP}
A. Carey and J. Phillips.
\newblock {Unbounded {F}redholm modules and spectral flow}
\newblock {\em Can. J. Math.} (4) \textbf{50} (1998), 673--718.

\bibitem{GS89}
E.~Getzler and A.~Szenes.
\newblock On the {C}hern character of a theta-summable {F}redholm module.
\newblock {\em J. Funct. Anal.} (2) \textbf{84} (1989)  343--357.

\bibitem{Hig06}
N.~Higson.
\newblock The residue index theorem of {C}onnes and {M}oscovici.
\newblock In {\em Surveys in noncommutative geometry}, volume~6 of {\em Clay
  Math. Proc.}, pages 71--126. Amer. Math. Soc., Providence, RI, 2006.

\bibitem{vNS21}
T.~D.~H.~van Nuland and A. Skripka.
\newblock {Spectral shift for relative Schatten class perturbations}.
\newblock arXiv:2102.00090 [math.FA].

\bibitem{Nak90}
M.~Nakahara.
\newblock {\em Geometry, Topology and Physics}.
\newblock IOP Publishing, 1990.

\bibitem{Pay07}
S.~Paycha.
\newblock (Second) Quantised resolvents and regularised traces.
\newblock {\em J. Geom. Phys.} (5) \textbf{57} (2007)  1345--1369.


\bibitem{PSS13}
D.~Potapov, A.~Skripka, and F.~Sukochev. 
\newblock Spectral shift function of higher order.
\newblock {\em Invent. Math.} (3) \textbf{193} (2013), 501--538.

\bibitem{Qui90}
D.~Quillen.
\newblock Chern-{S}imons forms and cyclic cohomology.
\newblock In {\em The interface of mathematics and particle physics ({O}xford,
  1988)}, volume~24 of {\em Inst. Math. Appl. Conf. Ser. New Ser.}, pages
  117--134. Oxford Univ. Press, New York, 1990.

\bibitem{Skr13}
A.~Skripka.
\newblock Asymptotic expansions for trace functionals.
\newblock {\em J. Funct. Anal.} (5) \textbf{266} (2014)  2845--2866.

\bibitem{ST19}
A. Skripka and A. Tomskova.
 {\it Multilinear Operator Integrals: Theory and Applications.} Lecture Notes in Math. 2250, Springer International Publishing, 2019, XI+192 pp.

\bibitem{Sui11}
W.~D. {\noopsort{suijlekom}}van Suijlekom.
\newblock Perturbations and operator trace functions.
\newblock {\em J. Funct. Anal.} (8) \textbf{260} (2011)  2483--2496.

\bibitem{Sui14}
W.~D. {\noopsort{suijlekom}}van Suijlekom.
\newblock {\em Noncommutative Geometry and Particle Physics}.
\newblock Springer, 2015.
\end{thebibliography}

\end{document}